\documentclass[a4paper,11pt]{amsart}
\pdfoutput=1
\usepackage[utf8]{inputenc}
\usepackage{amsmath,amssymb,amsfonts}
\usepackage{epsfig}
\usepackage{mathrsfs}
\usepackage{graphicx,color}
\usepackage{eucal}
\usepackage{bbm}

\newtheorem{thm}{Theorem}[section]

\newtheorem{cor}{Corollary}[section]
\newtheorem{rmk}{Remark}[section]
\newtheorem{lma}{Lemma}[section]
\newtheorem{assu}{Assumption}[section]
\newtheorem{exm}{Example}[section]

\newtheorem{alg}{Algorithm}[section]
\usepackage{url}
\usepackage[foot]{amsaddr}
\usepackage{array,multirow}

\newcommand{\norm}[1]{\left|  #1\right|}
\newcommand{\bnorm}[1]{\big| #1 \big|}

\newcommand{\Ex}{\mathbb{E}}
\newcommand{\ip}[1]{\left< #1 \right>}

\newcommand{\sm}[1]{\left[ \begin{smallmatrix} #1 \end{smallmatrix} \right]}
\newcommand{\bm}[1]{ \begin{bmatrix} #1 \end{bmatrix}}
\newcommand{\oX}{\overline{X}}
\newcommand{\uX}{\underline{X}}

\numberwithin{equation}{section}
\font\eka=cmex10
\usepackage{ae}
\def\ind{\mathrel{\hbox{\rlap{%
\hbox to 7.5pt{\hrulefill}}\raise6.6pt\hbox{\eka\char'167}}}}
\begin{document}
\title[Continuous-time GP dynamics in GRN inference]{Continuous time Gaussian process dynamical models in gene regulatory network inference*}

\author[Aalto]{Atte Aalto$^{1}$} 
\address{1: Luxembourg Centre for Systems Biomedicine, University of Luxembourg; 6 avenue du Swing; 4367 Belvaux; Luxembourg}
\email{atte.aalto@uni.lu}

\author[Viitasaari]{Lauri Viitasaari$^{2}$}
\address{2: Department of Mathematics and Statistics, University of Helsinki; P.O. Box 68, Gustaf H\"allstr\"omin katu 2b; 00014 Helsinki; Finland}
\email{lauri.viitasaari@helsinki.fi}

\author[Ilmonen]{Pauliina Ilmonen$^{3}$}
\address{3: Department of Mathematics and Systems Analysis, Aalto University School of Science; P.O. Box 11100; 00076 Aalto; Finland}
\email{pauliina.ilmonen@aalto.fi}

\author[Mombaerts]{Laurent Mombaerts$^{1}$}
\thanks{Funding: AA is funded by ERANET for Systems Biology ERASysApp and Luxembourg National Research Fund, project CropClock, grant reference INTER/SYSAPP/14/02.} 
\email{laurent.mombaerts@uni.lu}

\author[Gon\c{c}alves]{Jorge Gon\c{c}alves$^{1}$}
 \thanks{*This is a preprint version of an article published with title ``Gene regulatory network inference from sparsely sampled noisy data'' in \emph{Nature Communications}, 11: 3493 (2020). Please cite the journal-version instead of the arXiv preprint.}
\email{jorge.goncalves@uni.lu}

\begin{abstract}

One of the focus areas of modern scientific research is to reveal mysteries related to genes and their interactions. The dynamic interactions between genes can be encoded into a gene regulatory network (GRN), which can be used to gain understanding on the genetic mechanisms behind observable phenotypes. GRN inference from time series data has recently been a focus area of systems biology. Due to low sampling frequency of the data, this is a notoriously difficult problem. We tackle the challenge by introducing the so-called continuous-time Gaussian process dynamical model, based on Gaussian process framework that has gained popularity in nonlinear regression problems arising in machine learning. The model dynamics are governed by a stochastic differential equation, where the dynamics function is modelled as a Gaussian process. We prove the existence and uniqueness of solutions of the stochastic differential equation. We derive the probability distribution for the Euler discretised trajectories and establish the convergence of the discretisation. We develop a GRN inference method called BINGO, based on the developed framework. BINGO is based on MCMC sampling of trajectories of the GPDM and estimating the hyperparameters of the covariance function of the Gaussian process. Using benchmark data examples, we show that BINGO is superior in dealing with poor time resolution and it is computationally feasible. 



\end{abstract}

\maketitle

\section{Introduction}

In 2017, Jeffrey C. Hall, Michael Rosbash, and Michael W. Young were awarded the Nobel prize in physiology or medicine for their discoveries of molecular mechanisms controlling the circadian rhythm of plants. Indeed, one of the focus areas of modern scientific research is to reveal mysteries related to genes, their interactions, and their connection to observable phenotypes. Application areas of analysing interactions of genes are not limited to plants and their circadian clocks. In the field of biomedicine, for example, knowledge of genes and their interactions plays a crucial role in prevention and cure of diseases.

Interactions between genes are typically represented as a gene regulatory network (GRN) whose nodes correspond to different genes, and a directed edge denotes a direct causal effect of some gene on another gene. The usual problem statement is to infer the network topology from given gene expression data.
Classically, GRN inference has been based on analysing steady state data corresponding to gene knockout experiments, based on silencing one gene and observing changes in the steady state expressions of other genes. However, carrying out knockout experiments on a high number of genes is costly and technically infeasible. Moreover, for example circadian clocks of plants are oscillating systems, and in practice it can be difficult to determine whether a particular measurement corresponds to a steady state. In contrast, methods that infer GRNs from time series data can infer the networks on a more global level using data from few experiments. Therefore GRN inference from time series data has gained more and more attention recently. This article presents BINGO (Bayesian Inference of Networks using Gaussian prOcess dynamical models). BINGO is designed for network inference mainly from time series data, but steady state measurements can be straightforwardly implemented as well.

We model the time series data $\{y_j\}_{j=0}^m$ as samples from a continuous trajectory, 
\begin{equation} \label{eq:meas}
y_j=x(t_j)+v_j
\end{equation}
 where $v_j$ represents measurement noise. The continuous trajectory is assumed to satisfy a nonlinear stochastic differential equation
\begin{equation} \label{eq:sde_intro}
dx = f(x)dt + dw,
\end{equation}
where $w$ is some driving process noise. Here $x$ is an $\mathbb{R}^n$-valued function, and thus also $f$ is vector-valued, 
\[ f(x)=\sm{f_1(x_1,...,x_n) \\ \vdots\\f_n(x_1,...,x_n)}. \] 
 If function $f_j$ depends on $x_i$, in the corresponding GRN there is a link from gene $i$ to gene $j$. The task in GRN inference is to discover this regulatory interconnection structure between variables. It is known that the dynamics of one gene can only be influenced by few other genes, that is, a component $f_j$ only depends on some components of $x$.

What is characteristic to the GRN inference problem is that the data tends to be rather poor in terms of temporal resolution and the overall amount of data. Low temporal resolution has a deteriorating effect on network inference---in linear systems ($f(x
)=Ax$ in \eqref{eq:sde_intro}) this is illustrated by the fact that matrices $A$ and $e^{A\Delta T}$ do not share even approximatively the same sparsity pattern when $\Delta T$ is big. In addition, many methods use derivatives estimated directly from the time series data using difference approximations or curve fitting techniques. Using different techniques can significantly effect the results, and therefore avoiding such derivative approximations is desirable.

In this article, we introduce a continuous-time version of the so-called Gaussian process dynamical model (GPDM) \cite{GPDM_long}. Essentially, the dynamics function $f$ in \eqref{eq:sde_intro} is modelled as a Gaussian process with some covariance function \cite{GP_book}. This defines $x$ as a stochastic process. We prove existence and uniqueness of solutions of the corresponding differential equation, and the convergence of the Euler discretisation. We derive a probability distribution for the discretised trajectory, enabling direct MCMC sampling of realisations of this process. This way the derivative approximations are avoided. This trajectory sampling is illustrated in Figure~\ref{fig:traj_samples_intro}, showing samples from $p(x|\theta,Y) \propto p(y|x,\theta)p(x|\theta)$, where $\theta$ represents hyperparameters, and $Y$ comprises the time series measurements. Then $p(y|x,\theta)$ is the measurement model arising from \eqref{eq:meas}, and $p(x|\theta)$ is the probability distribution for the trajectory $x$, arising from \eqref{eq:sde_intro}. Network inference is then done by estimating the hyperparameters of the chosen covariance function. The technique of performing variable selection based on estimating variable-specific hyperparameters is known as \emph{automatic relevance determination} (ARD) \cite{MacKay,Neal_ARD}. In our approach, missing measurements as well as non-constant sampling frequency are easy to treat, and these functionalities are already implemented in the code, as well as a possibility to include prior information on the network. In the end, we are mainly interested in which of the variable-specific hyperparameters are non-zero. Therefore we introduce an indicator variable matrix for these hyperparameters, that can also be interpreted as the adjacency matrix of the GRN. The pipeline of BINGO is illustrated in Figure~\ref{fig:pipeline}. We demonstrate that BINGO is superior in dealing with poor time resolution while still remaining computationally feasible.

\begin{figure}[t]
\center
\includegraphics[width=12cm]{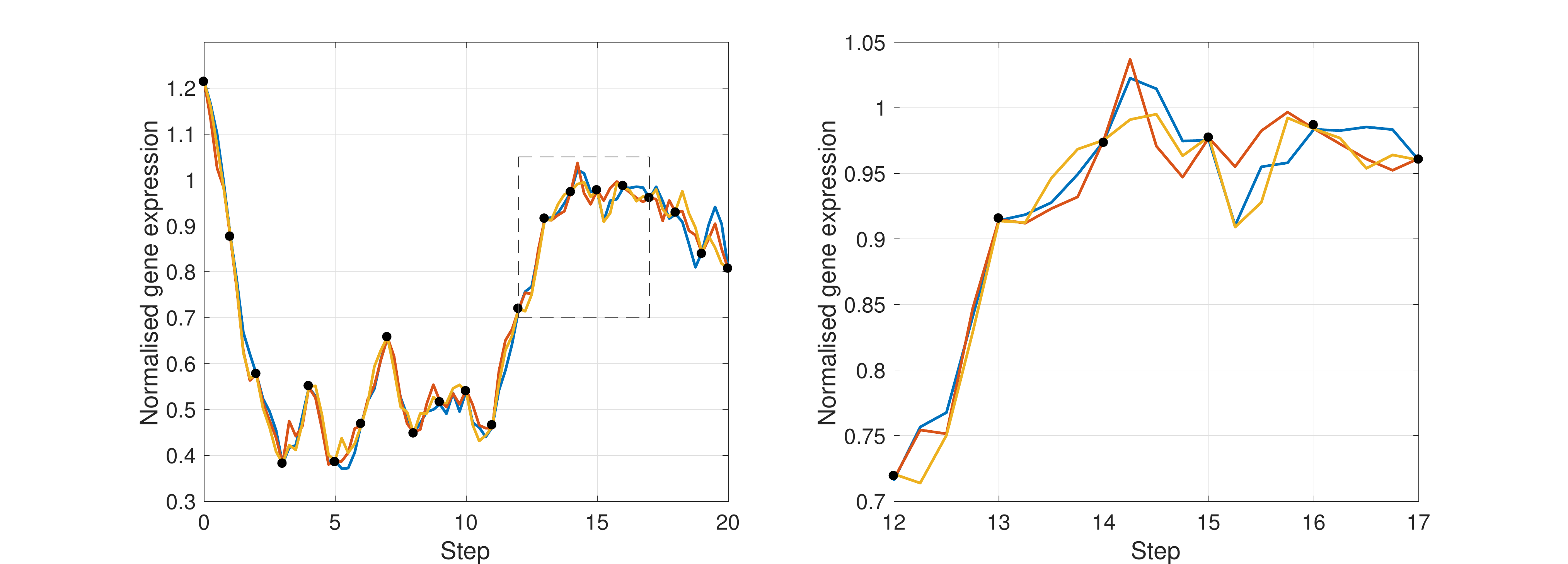} 
\caption{Three trajectory samples together with the discrete measurements from one of the DREAM4 datasets. The right panel is an enlargement of the rectangle on the left panel.}
\label{fig:traj_samples_intro}\end{figure}

Linear version of the method is presented in \cite{Bayesian_var}, where the MCMC samplers for the trajectory and the network topology are introduced. 
Other methods that model the dynamics function using a Gaussian process are presented in \cite{Aijo,Klemm,CSI_imp,Penfold_review}. The main difference between BINGO and these older methods is that they all treat the problem as a nonlinear regression problem with input-output pairs \cite{Aijo} 
\[ 
\left\{\left( y_{j-1} , \frac{y_{j}-y_{j-1}}{t_{j}-t_{j-1}} \right)\right\}_{j=1}^{m}
\]  or using some other method to approximate the derivatives from the time series data. As mentioned above, we avoid this derivative estimation by fitting continuous-time trajectories to the time series data. It should be noted that the GP framework can handle combinatorial effects, meaning nonlinearities that cannot be decomposed into $f(x_1,x_2)=f_1(x_1)+f_2(x_2)$. This is an important property for modelling a chemical system---such as gene expression---where reactions can happen due to combined effects of reactant species. For example, the dynamics corresponding to a chemical reaction $x_1+x_2 \to x_3$ cannot be modelled by $\dot x_3 = f_1(x_1)+f_2(x_2)$.

\begin{figure}
\mbox{\hspace{-28mm}
\includegraphics[width=18cm]{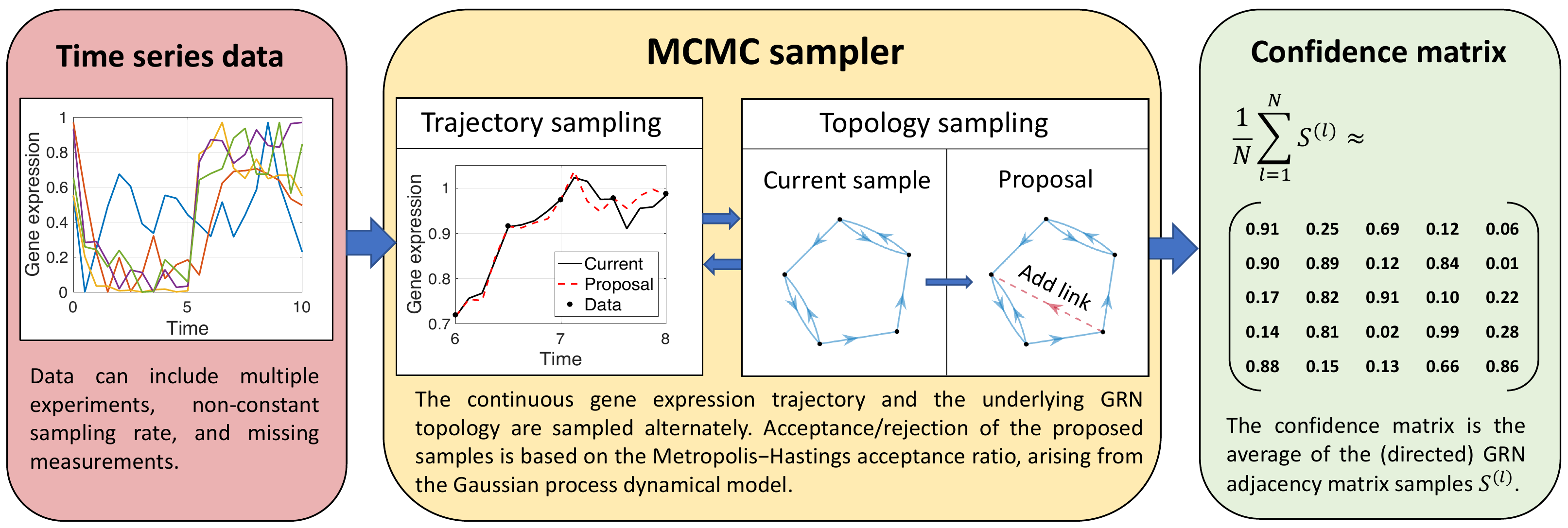}}
\caption{The BINGO pipeline: A proposal trajectory sample $\hat x$ is drawn by perturbing slightly the current sample $x^{(l)}$. The proposal is accepted ($x^{(l+1)}=\hat x$) or rejected ($x^{(l+1)}=x^{(l)}$) based on the Metropolis--Hastings acceptance ratio. A GRN topology proposal $\hat S$ is constructed by adding or removing one link to/from the current topology $S^{(l)}$.}
\label{fig:pipeline}
\end{figure}

Several GRN inference problems from different types of data have been posed as competitive challenges by the \emph{Dialogue for Reverse Engineering Assessments and Methods} (DREAM) project. Results and conclusions, as well as some top performers are introduced in \cite{GRN_review,DREAM_community}.  The inference problems are based mainly on two types of data, namely time series data, and steady state measurements corresponding to gene knockout or knockdown experiments. 
A review \cite{Penfold_review} on methods based on time series data concluded that methods based on nonparametric nonlinear differential equations performed best. Such methods include Gaussian process models 
and random forest models \cite{dynGENIE3}. Other types of ODE models tend to make rather restrictive assumptions on the dynamics, such as linear dynamics \cite{Bayesian_var,Kojima,TSNI}, or a library of nonlinear functions  \cite{SINDy,ARNI,Sparse_Id,SINDy_GRN}. Mechanistic models \cite{Aderhold_mechanistic,Oates_mechanistic} try to fit the data using dynamical systems constructed from enzyme kinetics equations.

The rest of the article is organised as follows. In Section~\ref{sec:GPDM}, we introduce the continuous-time Gaussian process dynamical models, and we prove the existence and uniqueness of solutions, and the convergence of the Euler discretisation. This is applied in Section~\ref{sec:distribution} where we derive the probability distribution for the discretised trajectory. The full network inference method BINGO is introduced in Section~\ref{sec:GRN} including incorporating several time series and knockout/knockdown experiments.
Section \ref{sec:results} is devoted to benchmark data experiments. We apply BINGO to the DREAM4 In Silico Network Challenge data \cite{GRN_review,DREAM_data,DREAM_models,DREAM_web} using either the time series data only, or including also the steady state experiments. The method is compared with the best performers in the DREAM4 challenge, as well as with more recent methods. Moreover, BINGO's performance with low sampling frequency is demonstrated by performing network inference on simulated data of the circadian clock of {\it Arabidopsis thaliana} \cite{Millar10}, using different sampling frequencies. Finally, to demonstrate BINGO's performance using real data, it has been applied to the IRMA in vivo dataset \cite{IRMA}, which is obtained from a synthetic network of five genes, and therefore the ground truth network is known. In Section~\ref{sec:discussion}, we provide a short summary and discuss future prospects.


\section{Continuous-time Gaussian process dynamical models} \label{sec:GPDM}

Discrete-time GPDMs (a.k.a. Gaussian process state space models \cite{Frigola_GPSS,GPSS_id}) were originally introduced in \cite{GPDM_long}, whose treatment was based on the GP latent variable models \cite{latent}. They are an effective tool for analysing time series data that is produced by a dynamical system that is unknown to us, or somehow too complicated to be presented using classical modelling techniques. In the original paper \cite{GPDM_long}, the method was used for human motion tracking from video data. Motion tracking problems remain the primary use of GPDMs \cite{GP_switch,GP_hand}, but other types of applications have emerged as well, such as speech analysis \cite{GPDM_speech}, traffic flow prediction \cite{GPtraffic}, and electric load prediction \cite{GP_elec}.

In this section we study theoretical properties of the continuous time GPDM trajectory defined as the solution $x \in \mathbb{R}^n$ on $t\in[0,T]$ for some fixed  $T$ to the stochastic differential equation 
\begin{equation} \label{eq:sde}
dx_t=f(u_t,x_t,\omega)dt+dw_t, \qquad x(0)=x_0,
\end{equation}
where the initial state $x_0$ is normally distributed, $x_0 \sim \mathcal{N}(m,P_0)$ for some covariance matrix $P_0$, $u_t$ is a smooth deterministic input function, and $w_t$ is an $n$-dimensional Brownian motion with diagonal covariance matrix $Q=\textup{diag}(q_1,...,q_n)$. Finally, $f=[f_1,...,f_n]^{\top}$ where each component $f_i=f_i(u,x,\omega)$ conditioned on a trajectory $x$ is modelled as a Gaussian process. For simplicity, we assume that each $f_i$ is centred (see Remark \ref{remark:non-centred}) and has a covariance $k_i$ depending on $u$ (input variable) and $x$ (state variable). That is, $\Ex f_i(u,x,\omega) = 0$ and $\Ex f_i(u,x,\omega)f_i(v,z,\omega) = k_i(u,v,x,z)$. 
\begin{rmk}
By Mercer's theorem, each covariance $k$ can be represented as 
$$
k(u,v,x,z) = \sum_{k=1}^\infty \lambda_k^2 \phi_k(u,x)\phi_k(v,z).
$$
Hence a Gaussian $f$ with covariance $k$ can be modelled by 
$$
f(u,x) = \sum_{k=1}^\infty \phi_k(u,x)\xi_k,
$$
where $\xi_k \sim \mathcal{N}(0,\lambda_k^2)$ are mutually independent. From this it is clear that for given $x$, $f(u,x)$ is Gaussian, whereas for random $x$, it is usually not.
\end{rmk}

Throughout the article we make the following assumption on the covariances $k_i$. 
\begin{assu}
\label{ass:cov}
For every $i=1,\ldots,n$, there exists a constant $L_i$ such that  
$$
|k_i(u_t,u_t,x,x)-k_i(u_t,u_t,x,z)| \leq L_i|x-z|^2
$$
uniformly in $t\in[0,T]$. 
\end{assu} 
\begin{exm}
The GRN inference algorithm developed below is based on the squared exponential covariance functions 
\begin{equation} \label{eq:sq_exp}
k_i(x,z)=\gamma_i \exp\left( -\sum_{j=1}^n \beta_{i,j} (x_j-z_j)^2 \right)
\end{equation}
and estimating the hyperparameters $\beta_{i,j}$. The hyperparameters satisfy $\gamma_i >0$ and $\beta_{i,j} \ge 0$. If $\beta_{i,j} >0$, it indicates that gene $j$ is a regulator of gene $i$. The Assumption \ref{ass:cov} is satisfied with the constant $L_i = \gamma_i\max_{1\leq j\leq n}\beta_{i,j}$. 
\end{exm}
Before stating and proving existence and uniqueness result for \eqref{eq:sde} we need one technical lemma. 
\begin{lma}
\label{lma:f_expectation_bound}
Suppose that Assumption \ref{ass:cov} holds and let $x,z\in \mathbb{R}^n$ be arbitrary. Then for any $p\geq 1$ there exists a constant $C$ depending on $p$ and the numbers $L_1,\ldots L_n$ such that 
$$
\Ex|f(u_t,x,\omega)-f(u_t,z,\omega)|^p \leq C|x-z|^p.
$$
\end{lma}
\begin{proof}
Since $f$ is a Gaussian vector, it suffices to prove the claim only for $p=2$. Furthermore, by triangle inequality, it suffices to prove that for each component $f_i$ we have
$$
\Ex|f_i(u_t,x,\omega)-f_i(u_t,z,\omega)|^2 \leq C|x-z|^2.
$$
Now 
$$
\Ex|f_i(u_t,x,\omega)-f_i(u_t,z,\omega)|^2 = k_i(u_t,u_t,x,x) + k_i(u_t,u_t,z,z) - 2k(u_t,u_t,x,z),
$$
and Assumption \ref{ass:cov} implies 
$$
\Ex|f_i(u_t,x,\omega)-f_i(u_t,z,\omega)|^2 \leq 2|x-z|^2
$$
which concludes the proof.
\end{proof}
\begin{cor}
\label{cor:continuity}
Suppose that Assumption \ref{ass:cov} holds and let $x,z\in \mathbb{R}^n$ be random variables. Then for any $p\geq 1$ there exists a constant $C$ depending on $p$ and the numbers $L_1,\ldots L_n$ such that 
$$
\Ex\left(|f(u_t,x,\omega)-f(u_t,z,\omega)|^p \, | \, x,z\right) \leq C|x-z|^p.
$$
\end{cor}
\begin{proof}
The claim follows from Lemma \ref{lma:f_expectation_bound} together with the fact that $f_i$ conditioned on $x$ and $z$ is Gaussian with covariance $k_i$.
\end{proof}
The following existence and uniqueness result for the stochastic differential equation \eqref{eq:sde} justifies the use of the continuous-time GPDM model.

\begin{thm}
\label{thm:existence}
Suppose that Assumption \ref{ass:cov} is satisfied. Then \eqref{eq:sde} admits a unique solution $x$. 
\end{thm}
\begin{proof}
We use Picard iteration and define
$$
x_t^0 = x_0,
$$
and for $j\geq 1$ we set
$$
x_t^j = x_0 + \int_0^t f(u_s,x_s^{j-1},\omega)ds + w_t - w_0.
$$
Then
$$
x_t^j-x_t^{j-1} = \int_0^t f(u_s,x_s^{j-1},\omega) - f(u_s,x_s^{j-2},\omega)ds
$$
and 
\begin{equation}
\label{eq:as-bound}
| x_t^j-x_t^{j-1} | \leq \int_0^t | f(u_s,x_s^{j-1},\omega) - f(u_s,x_s^{j-2},\omega)| ds.
\end{equation} 
Taking expectation, conditioning, and using Corollary \ref{cor:continuity} then gives
\begin{equation}
\label{eq:recursion}
\Ex| x_t^j-x_t^{j-1} | \leq C\int_0^t \Ex|x_s^{j-1}-x_s^{j-2}|ds.
\end{equation} 
We now claim that 
$$
\Ex| x_t^j-x_t^{j-1} | \leq \frac{C_1C^jt^j}{j!} + \frac{C_2C^{j-1}t^{j-1}}{(j-1)!}.
$$
This follows by induction. For $j=1$ we have
$$
| x_t^1-x_0 | =  \left|\int_0^t f(u_s,x_0)ds + w_t - w_0\right|
\leq \sup_{s \in [0,T]}| f(u_s,x_0)| t + \sup_{s\in[0,T]}|w_t|
$$
which proves the claim for $j=1$ as the supremum of Gaussian process $f$ and the supremum of $w_t$ have all moments finite.
Suppose
$$
\Ex| x_s^j-x_s^{j-1}| \leq \frac{C_1 C^j s^j}{j!} + \frac{C_2 C^{j-1}s^{j-1}}{(j-1)!}.
$$
Then \eqref{eq:recursion} implies
$$
\Ex| x_t^{j+1}-x_t^{j} | \leq \int_0^t \frac{C_1 C^{j+1} s^j}{j!} + \frac{C_2C^{j}s^{j}}{(j-1)!} ds = \frac{C_1 C^{j+1} t^{j+1}}{(j+1)!} + \frac{C_2 C^{j}t^{j}}{j!}.
$$
In particular, this gives 
$$
\sup_{t\in[0,T]}\Ex| x_t^{j+1}-x_t^{j} | \leq \frac{C_1 C^{j+1} T^{j+1}}{(j+1)!} + \frac{C_2 C^{j}T^{j}}{j!} \rightarrow 0
$$
and
$$
\sum_{j=0}^\infty \sup_{t\in[0,T]}\Ex| x_t^{j+1}-x_t^{j} | < \infty.
$$
On the other hand, from \eqref{eq:as-bound} we get 
$$
\sup_{t\in[0,T]}| x_t^j-x_t^{j-1} | \leq \int_0^T | f(u_s,x_s^{j-1},\omega) - f(u_s,x_s^{j-2},\omega)| ds.
$$
Consequently, taking expectation gives
$$
\Ex \left[\sup_{t\in[0,T]}| x_t^j-x_t^{j-1} |\right] \leq CT\sup_{t\in[0,T]}\Ex| x_t^{j-1}-x_t^{j-2} |,
$$
and thus we also have
$$
\sum_{j=0}^\infty \Ex\left[\sup_{t\in[0,T]}| x_t^{j+1}-x_t^{j} |\right] < \infty.
$$
This implies that 
$$
\sum_{j=0}^k (x_t^{j+1}-x_t^{j}) = x_t^{k+1}-x_0
$$
converges uniformly to an integrable random variable. Finally, since $f(u_s,x,\omega)$ is continuous in $x$ by Gaussianity and Lemma \ref{lma:f_expectation_bound}, we observe that the limit $x = \lim_{j\to\infty} x^j$ satisfies \eqref{eq:sde}. 
\end{proof}
\begin{rmk}
\label{remark:non-centred}
We stress that while we assumed the Gaussian process $f$ to be centred for the sake of simplicity, the extension to a non-centred case is rather straightforward. Indeed, if for each component $f_i$ the mean function $\Ex f_i(u_t,x,\omega) = m_i(u_t,x)$ is Lipschitz continuous with respect to $x$ uniformly in $t$, i.e.
$$
|m_i(u_t,x) - m_i(u_t,z)| \leq L|x-z|,
$$ 
then the existence and uniqueness follows from the above proof by centering $f$ first. 
We leave the details to the reader.
\end{rmk}
The following result studies the basic properties of the solution. 
\begin{thm}
Suppose that Assumption \ref{ass:cov} holds. Then the solution $x$ to \eqref{eq:sde} is H\"older continuous of any order $\gamma < \frac12$. Furthermore, $\sup_{t\in[0,T]}|x_t|$ has all the moments finite. 
\end{thm}
\begin{proof}
Clearly, each $x^j$ in the proof of Theorem \ref{thm:existence} is continuous. Consequently, the solution $x$ is continuous as a uniform limit of continuous trajectories. The H\"older continuity then follows from \eqref{eq:sde} and the H\"older continuity of the Brownian motion $w$. Indeed, since $f(u_s,x,\omega)$ is continuous in $x$ and $x$ is bounded as a continuous function on a bounded interval $[0,T]$, it follows that $f(u_s,x_s,\omega)$ is also bounded. Finally, the existence of all moments follow from the fact that $f(u_s,x_s,\omega)$ has all the moments finite as well as $\sup_{t\in[0,T]}|w_t|$ has all the moments finite. 
\end{proof}

The method's numerical implementation will be based on the Euler discretised equation \eqref{eq:sde}. Define therefore a partition $\pi^M = \{0=\tau_0<\tau_1<\ldots<\tau_M=T\}$ of the compact interval of interest $[0,T]$. Denote the discretised trajectory corresponding to the partition $\pi^M$ by $X_{\tau_k}^M$, and recall that its dynamics are given by
\begin{equation} \label{eq:euler}
X_{\tau_k}^M=X_{\tau_{k-1}}^M+\delta \tau_kf(u_{\tau_{k-1}},X_{\tau_{k-1}}^M,\omega)+w_{\tau_k}-w_{\tau_{k-1}}
\end{equation}
where $\delta \tau_k := \tau_k-\tau_{k-1}$, and $k=1,...,M$. Later, we will obtain a probability distribution for the discrete trajectory $X=[X_{\tau_0},X_{\tau_1},...X_{\tau_M}]$, but first we show the pointwise (in $\omega$) convergence to the continuous solution of \eqref{eq:sde} as the temporal discretisation is refined.

We study the continuous version defined for $t\in[\tau_{k-1},\tau_k]$ by
\begin{equation}
\label{eq:continuous-Euler}
\overline{X}_{t}^M = \overline{X}_{\tau_{k-1}}^M + (t-\tau_{k-1})f(u_{\tau_{k-1}},\overline{X}_{\tau_{k-1}}^M,\omega) + w_t - w_{\tau_{k-1}}.
\end{equation}
Note that $X_{\tau_k}^M = \overline{X}_{\tau_k}^M$ for all $k$. 
\begin{thm}
Suppose that Assumption \ref{ass:cov} holds and consider arbitrary discretisation such that $\sup_M|\pi^M|M < \infty$, where $|\pi^M| = \max_k(\tau_k-\tau_{k-1})$. 
Then for any $p\geq 1$
$$
\Ex\left[\sup_{t\in[0,T]}|x_t - \overline{X}_t^M|\right]^p \leq C|\pi^M|^p.
$$
Moreover, for any $\epsilon>0$ we have
$$
\sup_{t\in[0,T]}|x_t - \overline{X}_t^M| \leq C|\pi^M|^{1-\epsilon}.
$$
almost surely. 
\end{thm}
\begin{proof}
Let $t \in [\tau_{k-1},\tau_k]$ and denote
$$
z_k = \left\Vert\sup_{t\in[\tau_{k-1},\tau_k]}| x_t - \overline{X}_t^M|\right\Vert_p,
$$
where $\Vert \cdot\Vert_p$ denotes the $p$-norm.
Now
$$
x_t - \overline{X}_t^M = x_{\tau_{k-1}}-\overline{X}_{\tau_{k-1}}^M + \int_{\tau_{k-1}}^t f(u_s,x_s,\omega) - f(u_{\tau_{k-1}},\overline{X}^M_{\tau_{k-1}},\omega)ds.
$$
As in the proof of Theorem \ref{thm:existence}, this implies
$$
z_k \leq z_{k-1} + \int_{\tau_{k-1}}^{\tau_k} \left\Vert f(u_s,x_s,\omega) - f(u_{\tau_{k-1}},\overline{X}^M_{\tau_{k-1}},\omega)\right\Vert_p ds \leq z_{k-1} + C z_k|\pi^M|.
$$
Let now $M$ be large enough such that $C|\pi^M| < 1$. We get
$$
(1-C|\pi^M|)z_k \leq z_{k-1}
$$
or equivalently
$$
z_k \leq \frac{1}{1-C|\pi^M|}z_{k-1}.
$$
Iterating then gives
$$
z_k \leq \left(\frac{1}{1-C|\pi^M|}\right)^k z_{1}=\left(1+\frac{C}{|\pi^M|^{-1}-C}\right)^k z_1.
$$
Note next that 
$$
\left(1+\frac{C}{|\pi^M|^{-1}-C}\right)^k \leq \left(1+\frac{\tilde{C}}{|\pi^M|^{-1}}\right)^M =\left(1+\frac{\tilde{C}}{|\pi^M|^{-1}}\right)^{|\pi^M|^{-1}|\pi^M|M} 
$$
for some other constant $\tilde{C}$. Since
$$
\left(1+\frac{\tilde{C}}{|\pi^M|^{-1}}\right)^{|\pi^M|^{-1}} \to e^{\tilde{C}}
$$
as $M\to \infty$
and $|\pi^M|M$ is bounded by assumption, it follows that 
$$
z_k \leq Cz_1
$$
for some unimportant constant $C$. But now 
$$
|x_t - \overline{X}_t^M| \leq \int_{0}^{\tau_1} |f(u_s,x_s,\omega) - f(u_{\tau_{k-1}},\overline{X}^M_{\tau_{k-1}},\omega)|ds \leq 2\sup_{s\in[0,T]}| f(u_s,x_s,\omega)||\pi^M|
$$
for $t\in[0,\tau_1]$ from which it follows that $z_1 \leq C|\pi^M|$ proving the first claim. Finally, the second claim is a direct consequence of the Borel-Cantelli lemma.
\end{proof}

\section{Probability distribution of the discretised trajectory} \label{sec:distribution}

The probability distribution $p(X|\theta)$ is now derived for the discrete trajectory $X=X^M$, where $\theta$ denotes collectively all the hy\-per\-parameters. The discretisation level index $M$ is dropped now, since we only use one discretisation level from now on. It holds that
\begin{equation} \label{eq:pXint}
p(X|\theta)=\int p(X|f,\theta)p(f|\theta)df.
\end{equation}
For given $f$, the trajectory $X$ is a Markov process, and therefore its distribution satisfies
\[
p(X|f,\theta)=p(X_{\tau_0}|\theta)\prod_{k=1}^M p(X_{\tau_k}|X_{\tau_{k-1}},f,\theta).
\]

Let us introduce notation $\oX:=[X_{\tau_1},\dots,X_{\tau_M}]^{\top}$ and $\uX:=[X_{\tau_0},\dots,X_{\tau_{M-1}}]^{\top}$. Same notation is also used for the different dimensions of the trajectory.
Then it holds that
\begin{align*}
&p(X|f,\theta) \\ & = \frac{p(X_{\tau_0}|\theta)}{(2\pi)^{Mn/2}|Q|^{M/2}}\prod_{k=1}^M\frac1{\delta \tau_k^{n/2}} \exp\left( -\frac1{2\delta \tau_k}  \norm{X_{\tau_k}-X_{\tau_{k-1}}-\delta \tau_kf(X_{\tau_{k-1}})}_{Q^{-1}}^2 \right)\\ 
&=\frac{p(X_{\tau_0}|\theta)}{(2\pi)^{Mn/2}|Q|^{M/2}|\Delta \tau|^{n/2}}\exp\left( - \sum_{k=1}^M\frac1{2\delta \tau_k}\norm{X_{\tau_k}-X_{\tau_{k-1}}-\delta \tau_kf(X_{\tau_{k-1}})}_{Q^{-1}}^2 \right) \\ 
&= \frac{p(X_{\tau_0}|\theta)}{(2\pi)^{Mn/2}|Q|^{M/2}|\Delta \tau|^{n/2}} \prod_{i=1}^n\exp\left(-\frac1{2q_i} \norm{ \oX_i-\uX_i-\Delta \tau f_i(\uX) }_{\Delta \tau^{-1}}^2 \right) 
\end{align*}
where $\Delta \tau$ is a diagonal matrix whose element $(k,k)$ is $\delta \tau_k$, and $f_i(\uX)=\big[ f_i(X_{\tau_0}),\dots,f_i(X_{\tau_{M-1}})\big]^{\top}$. 

Now $p(X|f,\theta)$ in the integral \eqref{eq:pXint} depends only on the values of $f$ at points $\uX$. By definition of a Gaussian process, the integral can equivalently be computed over a collection of finite-dimensional, normally distributed random variables $F=[F_1,\dots,F_n] \in \mathbb{R}^{M \times n}$ where $F_i \in \mathbb{R}^{M}$ has mean zero, and covariance $K_i(\uX)$ given elementwise by $[K_i(\uX)]_{j,k} = k_i(X_{\tau_{j-1}},X_{\tau_{k-1}})$. The integral in \eqref{eq:pXint} can be computed analytically (see Appendix~\ref{app:expint}),
\begin{align*}
& \int p(X|f,\theta)p(f|\theta)df \\
& =\frac{p(X_{\tau_0}|\theta)}{(2\pi)^{Mn}|Q|^{M/2}|\Delta \tau|^{n/2}} \\ &  \prod_{i=1}^n\int \frac1{|K_i(\uX)|^{1/2}}\exp\left(-\frac1{2q_i} \norm{\oX_i-\uX_i-\Delta\tau F_i}_{\Delta\tau^{-1}}^2-\frac12 \norm{F_i}_{K(\uX)^{-1}}^2\right) dF_i \\
& = \frac{p(X_{\tau_0}|\theta)}{(2\pi)^{Mn/2}|Q|^{M/2}|\Delta \tau|^{n/2}}\prod_{i=1}^n \frac1{|K_i(\uX)|^{1/2}\left| \frac{\Delta\tau}{q_i}+K_i(\uX)^{-1}\right|^{1/2}} \\ & \exp\left(-\frac1{2q_i}\norm{\oX_i-\uX_i}_{\Delta\tau^{-1}}^2 + \frac1{2q_i^2} (\oX_i-\uX_i)^{\top}\left(\frac{\Delta\tau}{q_i}+K_i(\uX)\right)^{-1}(\oX_i-\uX_i) \right). 
\end{align*}
Applying the Woodbury identity to the exponent gives
\[
(q_i\Delta\tau)^{-1}-\frac1{q_i}\Big( \Delta\tau ( q_i\Delta\tau)^{-1}\Delta\tau +K_i(\uX)^{-1} \Big)\frac1{q_i}=\left( \Delta\tau K_i(\uX)\Delta\tau + q_i\Delta\tau \right)^{-1},
\]
and the determinant lemma gives (recall $Q$ is a diagonal matrix with $q_i$'s on the diagonal) 
\begin{align*}
&|Q|^{M/2}|\Delta\tau|^{n/2} \prod_{i=1}^n |K_i(\uX)|^{1/2}\left| \frac{\Delta\tau}{q_i}+K_i(\uX)^{-1}\right|^{1/2} \\ & =\prod_{i=1}^n |q_i\Delta\tau||K_i(\uX)|^{1/2} \left|\Delta\tau(q_i\Delta\tau)^{-1}\Delta\tau + K_i(\uX)^{-1}\right| \\
&= \prod_{i=1}^n \big|\Delta \tau K_i(\uX)\Delta \tau+q_i\Delta\tau \big|^{1/2}.
\end{align*}
Finally, the desired probability distribution is
\begin{align} \label{eq:pX}
&p(X|\theta)  = 
\frac{p(X_{\tau_0}|\theta)}{(2\pi)^{Mn/2}} \prod_{i=1}^n \frac1{|\Delta \tau K_i(\uX)\Delta \tau+q_i\Delta \tau|^{1/2}} \\ \nonumber & \qquad \cdot \exp\left(-\frac1{2} (\oX_i-\uX_i)^{\top}\left(\Delta \tau K_i(\uX)\Delta \tau + q_i \Delta \tau \right)^{-1}(\oX_i-\uX_i)  \right).
\end{align}

Note that above it was implicitly assumed that the covariance $K_i(\uX)$ is positive definite. This assumption is only violated if $X_{\tau_j}=X_{\tau_k}$ for some $j \ne k$ or if the covariance function $k_i$ is degenerate. In this case the integral should be computed over a lower-dimensional variable $F_i$, but the end result would not change.

Note also that \eqref{eq:pX} corresponds to the finite dimensional distribution of the continuous Euler scheme \eqref{eq:continuous-Euler} evaluated at discretisation points. Since \eqref{eq:continuous-Euler} converges strongly to the solution $x$ of \eqref{eq:sde}, the finite dimensional distributions converge as well. This means that \eqref{eq:pX} is a finite dimensional approximation of the distribution of $x$.

\section{Network inference method} \label{sec:GRN}

Consider then the original problem, that is, estimating the hyperparameters from given time series data. Denote $Y=[y_0,y_1,...,y_m]$ where $y_j$ is assumed to be a noisy sample from the continuous trajectory $x$, that is, $y_j=x(t_j)+v_j$, and $v_j$ is a Gaussian noise with zero mean and covariance $R=\textup{diag}(r)$, and $v_j \perp v_k$ when $j \ne k$. We intend to draw samples from the parameter posterior distribution using an MCMC scheme. Therefore, we only need the posterior distribution up to constant multiplication. Denoting the hyperparameters collectively by $\theta$, the hyperparameter posterior distribution is
\[
p(\theta|Y) \propto p(Y,\theta) = \int p(Y,x,\theta) dx = \int p(Y|x,\theta)p(x|\theta)p(\theta)dx.
\]
Here $p(Y|x,\theta)$ is the Gaussian measurement error distribution, $p(x|\theta)$ will be approximated by \eqref{eq:pX} for the discretised trajectory $X$, and $p(\theta)$ is a prior for the hyperparameters.  This prior consists of independent priors for each parameter. The integration with respect to the trajectory $x$ is done by MCMC sampling. In the network inference algorithm, we consider only the squared exponential covariance function \eqref{eq:sq_exp}. The function $f_i$ has mean
\[
m_i(x)=b_i - a_i x_i
\]
where $a_i$ and $b_i$ are regarded as nonnegative hyperparameters corresponding to basal transcription ($b_i$) and mRNA degradation ($a_i$).

For sampling the hyperparameters $\beta_{i,j}$ in the squared exponential covariance function \eqref{eq:sq_exp}, we introduce an indicator variable as in \cite{Bayesian_var}. That is, each hyperparameter is a product $\beta_{i,j}=S_{i,j}H_{i,j}$, where $S_{i,j} \in \{0,1\}$ and $H_{i,j} \ge 0$. The state of the sampler consists of the indicator variable $S$, the hyperparameters ($i,j=1,...,n$) $H_{i,j}$, $\gamma_i$, $r_i$, $q_i$, $a_i$, $b_i$ and the discrete trajectory $X$. They are sampled using a Gibbs sampler (or more precisely, Metropolis--Hastings within Gibbs sampler) as described below.

For the Gibbs sampler, notice that $p(X|\theta)$ given in \eqref{eq:pX} is readily factorised in form
\begin{equation} \label{eq:factor}
p(X|\theta)=\frac{p(X_{\tau_0}|\theta)}{(2\pi)^{Mn/2}} \prod_{i=1}^n P_i(S_i,H_i,\gamma_i,q_i,a_i,b_i,X).
\end{equation}
This factorisation makes it natural to sample $S$, $H$, $\gamma$, $a=\{a_1,...,a_n\}$, and $b=\{b_1,...,b_n\}$ one dimension at a time. However, each factor $P_i$ still depends on the full trajectory $X$, so the trajectory sampling is done separately. Also, when using the Crank--Nicolson sampling (see Section~\ref{sec:CN}), the sampling of $q$ is intertwined with the trajectory sampling, so they are sampled together. This two-phase sampling scheme is described in the following algorithm. Here the algorithm is presented in its basic form. Some ways to make the sampling more efficient are presented in Appendix~\ref{app:sampling}. We assume that the initial time $\tau_0$ coincides with the time of the first measurement $t_0$, so that $p(X_{\tau_0}|\theta)=\mathcal{N}(y_0,R)$. In the algorithm, this is included in the data fit term $p(Y|x,\theta)$.

\begin{alg} Denote the $l^{\textup{th}}$ samples by parenthesised superindex, e.g., $X^{(l)}$ is the trajectory of the $l^{\textup{th}}$ sample. The proposal samples are denoted by a hat.

\medskip

\noindent {\bf Indicator and hyperparameter sampling:}

 For $i=1,...,n$:
\begin{itemize}
\item Sample the $i^{\textup{th}}$ row of $S$ by drawing $\hat j$ from uniform distribution over $\{1,...,n\}$. Then
\[
\hat S_{i,j} =  \left\{ \begin{array}{ll} S_{i,j}^{(l)}, & \textup{if } j \ne \hat j, \\ 1-S_{i,j}^{(l)}, & \textup{if } j = \hat j. \end{array} \right.
\]
\item Sample $H_i=[H_{i,1},...,H_{i,n}]$, $\gamma_i$, $a_i$, and $b_i$ using random walk sampling, that is, add small changes to each component, drawn from zero-mean normal distribution. If the proposal sample is negative, take its absolute value.

\item Accept the proposal samples with probability
\[
\frac{P_i(\hat S_i, \hat H_i, \hat \gamma_i, q_i^{(l)}, \hat a_i,\hat b_i, X^{(l)})p(\hat S_i, \hat H_i, \hat \gamma_i, q_i^{(l)},\hat a_i,\hat b_i)}{P_i(S_i^{(l)}, H_i^{(l)}, \gamma_i^{(l)}, q_i^{(l)},a_i^{(l)},b_i^{(l)}, X^{(l)})p(S_i^{(l)},H_i^{(l)},\gamma_i^{(l)},q_i^{(l)},a_i^{(l)},b_i^{(l)})}
\]
where $p$ is the hyperparameter prior, and the factors $P_i$ are defined in \eqref{eq:factor}.

\item Sample $\hat R$ with random walk sampling, with acceptance probability
\begin{align*}
&\frac{p(\hat R)|R^{(l)}|^{(m+1)/2}}{p(R^{(l)})|\hat R|^{(m+1)/2}}  \exp\left(\frac12 \sum_{j=0}^m \bnorm{y_j-X^{(l)}C_j}_{(R^{(l)})^{-1}}^2-\bnorm{y_j-X^{(l)}C_j}_{\hat R^{-1}}^2\right) 
\end{align*}
\end{itemize}

\medskip

\noindent {\bf Trajectory sampling:} 

\begin{itemize}

\item Sample $\hat X_i = X_i^{(l)}+Bg$, where $g \sim \mathcal{N}(0,\varepsilon I)$, $B=[b_1,...,b_{2m_b}]$, and 
\[
b_j=\left\{ \begin{array}{ll} \frac1j\bm{ \sin\left( \frac{2\pi j \tau_0}{T}\right),...,\sin\left( \frac{2\pi j \tau_M}{T}\right)}^{\top}, & j=1,...,m_b, \\ \frac1{j-m_b}\bm{\cos\left( \frac{2\pi (j-m_b) \tau_0}{T}\right),...,\cos\left( \frac{2\pi (j-m_b) \tau_M}{T}\right)}^{\top}, & j=m_b+1,...,2m_b, \end{array} \right.
\]
where $m_b=\lfloor M/2 \rfloor$.
\item Sample $\hat Q$ using the random walk sampling.

\item Accept $\hat X$ and $\hat Q$ with probability
\begin{align*}
&\frac{p(\hat Q)}{p(Q^{(l)})}  \exp\left(\frac12 \sum_{j=0}^m \bnorm{y_j-X^{(l)}C_j}_{(R^{(l+1)})^{-1}}^2-\bnorm{y_j-\hat XC_j}_{(R^{(l+1)})^{-1}}^2\right) \\ &   \times \prod_{i=1}^n\frac{P_i(S_i^{(l+1)}, H_i^{(l+1)}, \gamma_i^{(l+1)}, \hat q_i, a_i^{(l+1)},b_i^{(l+1)} \hat X)}{P_i(S_i^{(l+1)}, H_i^{(l+1)}, \gamma_i^{(l+1)}, q_i^{(l)}, a_i^{(l+1)},b_i^{(l+1)}, X^{(l)})}
\end{align*}
where $C_j \in \mathbb{R}^{(M+1)\times 1}$ gives the element from the full trajectory $X$ corresponding to the measurement $y_j$. In the case $\{t_0,...,t_m\} \subset \{ \tau_0,...,\tau_M\}$, $C_j$ is a vector with one at position $k$ satisfying $t_j=\tau_k$, and zeros elsewhere. 
\end{itemize}
\end{alg}

The algorithm contains a burn-in period, and additional thinning, that is, not every sample $l$ is collected. The output of the algorithm is the average of the indicator variable samples, which converges as  the number of samples increases:
$$
\frac1{N_{\textup{sample}}}\sum_{l=1}^{N_{\textup{sample}}} S^{(l)} \to \Ex (S|Y).
$$
The element $(i,j)$ of this matrix gives the probability that $\beta_{i,j}$ is not zero.

Prior probability distributions for different hyperparameters are described in Appendix~\ref{sec:implementation}. For $S$ we use $p(S) \propto \eta^{|S|_0}$ where $|S|_0$ gives the number of ones in $S$, and the parameter $\eta >0$ can be set to obtain a desired sparsity level for the solution. This prior means that the existence of a link is independent of other links, and the prior probability for the existence of any given link is $\frac{\eta}{1+\eta}$. A default value $\eta=1/n$ was used in all experiments of this article.

\subsection{Incorporation of several time series and knockout/knockdown experiments} \label{sec:other_data}

Several time series experiments can be easily incorporated. For fixed $f$, the probability distributions for different time series are independent. In the end, this leads to the same format of the probability distribution \eqref{eq:pX}, but the trajectories are concatenated. Then $\oX$ contains the concatenated trajectories, except for the first point in each separate discretised trajectory, and $\uX$ contains all trajectories, except for the last points in each trajectory.

In a knockout experiment a particular gene is ``de-activated", meaning that its expression is artificially put to zero. From an experiment where gene $i$ has been knocked out, it is not possible to deduce anything about $f_i$, since the dynamics of the $i^{\textup{th}}$ gene are artificially tampered with. Therefore these experiments are excluded from the cost functions corresponding to $f_i$.

In a steady state experiment, the system is allowed to evolve a long time without any excitation, so that it finally attains a steady state, where it should hold that $f(x_{ss})=0$. In the method, some noise is added to steady state measurements, and therefore, at a steady state point $x_{ss}$, it is assumed that $f_i(x_{ss})=v_{i,ss}$, where $v_{i,ss} \sim \mathcal{N}(0,M_{ss})$. The incorporation of the steady state data to \eqref{eq:pX} is done by replacing $K_i(\uX)$, $\oX_i-\uX_i$, $\Delta\tau$, and $q_iI$ by
\[
K_i([\uX,X_{ss}]), \quad \bm {\oX_i - \uX_i \\ 0}, \quad   \bm{\Delta\tau & \\ & I}, \quad \textup{and} \quad \bm{q_iI & \\ & M_{i,ss}I},
\]
respectively.

A steady state experiment can also be a knockout experiment. At the steady state $z_i$ corresponding to knockout of gene $i$, it should hold that $f_j(z_i)=0$ for all $j$, except $j=i$, since the dynamics of gene $i$ have been artificially tampered with.

A gene knockdown experiment is similar to a gene knockout experiment, but the genes are only repressed instead of completely inactivated, and it is taken into account in exactly the same way as a knockout experiment.

When using all of the knockout and knockdown steady state data, we assume that there is one point $x_{ss}$ where $f_i(x_{ss})=0$ for all $i$. This steady state value is sampled, and its prior is a normal distribution whose mean is the sample mean of all steady state measurements including the actual steady state measurement, knockout measurements, knockdown measurements, and the multifactorial data (in the DREAM4 10-gene challenge). The covariance of the prior distribution of $x_{ss}$ is the sample covariance of this data, divided by the number of the steady state measurements. This corresponds to the sample covariance of the mean. We assume that at the steady state, it holds that $f_i(x_{ss})=v_{i,ss}$ where $v_{i,ss} \sim N(0,M_{i,ss})$, and at the knockout and knockdown points $f_i(x_{j,ko}) = v_{i,ko}$ where $v_{i,ko} \sim N(0,M_{i,ko})$. Also the covariances $M_{i,ss}$ and $M_{i,ko}$ are sampled, and they are given noninformative inverse gamma prior distributions. The incorporation of the knockout/knockdown data to $p(X|\theta)$ in \eqref{eq:pX} is done by 
replacing $\oX_i-\uX_i$, $\Delta\tau$, and $q_iI$ by
\[
\bm {\oX_i - \uX_i \\ 0}, \quad   \bm{\Delta\tau & \\ & I}, \quad \textup{and} \quad \bm{q_iI & & \\ & M_{i,ss} & \\ & & M_{i,ko}I},
\]
respectively,
and $K_i(\uX)$ is replaced by $K_i([\uX,x_{ss},y_{i,ko/kd}])$ where $y_{i,ko/kd}$ denotes the collection of all knockout/knockdown measurements except for the ko/kd of gene $i$.

\section{Benchmark data examples} \label{sec:results}

BINGO has been benchmarked using the data from the DREAM4 {in silico} network challenge,  simulated data from the circadian clock of the plant {\it Arabidopsis thaliana} with varying sampling rate and noise levels, as well as the IRMA {in vivo} dataset. In all the experiments, BINGO is compared with three new methods, dynGENIE3 \cite{dynGENIE3}, iCheMA \cite{Aderhold_mechanistic}, and ARNI \cite{ARNI}. They are designed for inference from time series data. In addition, DREAM4 and IRMA datasets have been used in benchmarking other methods, and some results found in the literature have been included in the comparison. Standard classifier scores are used for the comparison, namely the area under the receiver operating characteristic curve (AUROC) and the area under the precision-recall curve (AUPR). Self-regulation is always excluded as in the DREAM4 challenge. Results of the benchmark cases are illustrated in Figure~\ref{fig:results}, and discussed below.

\begin{figure}[t]
\center 
\mbox{\hspace{-27mm}
\includegraphics[width=4.5cm]{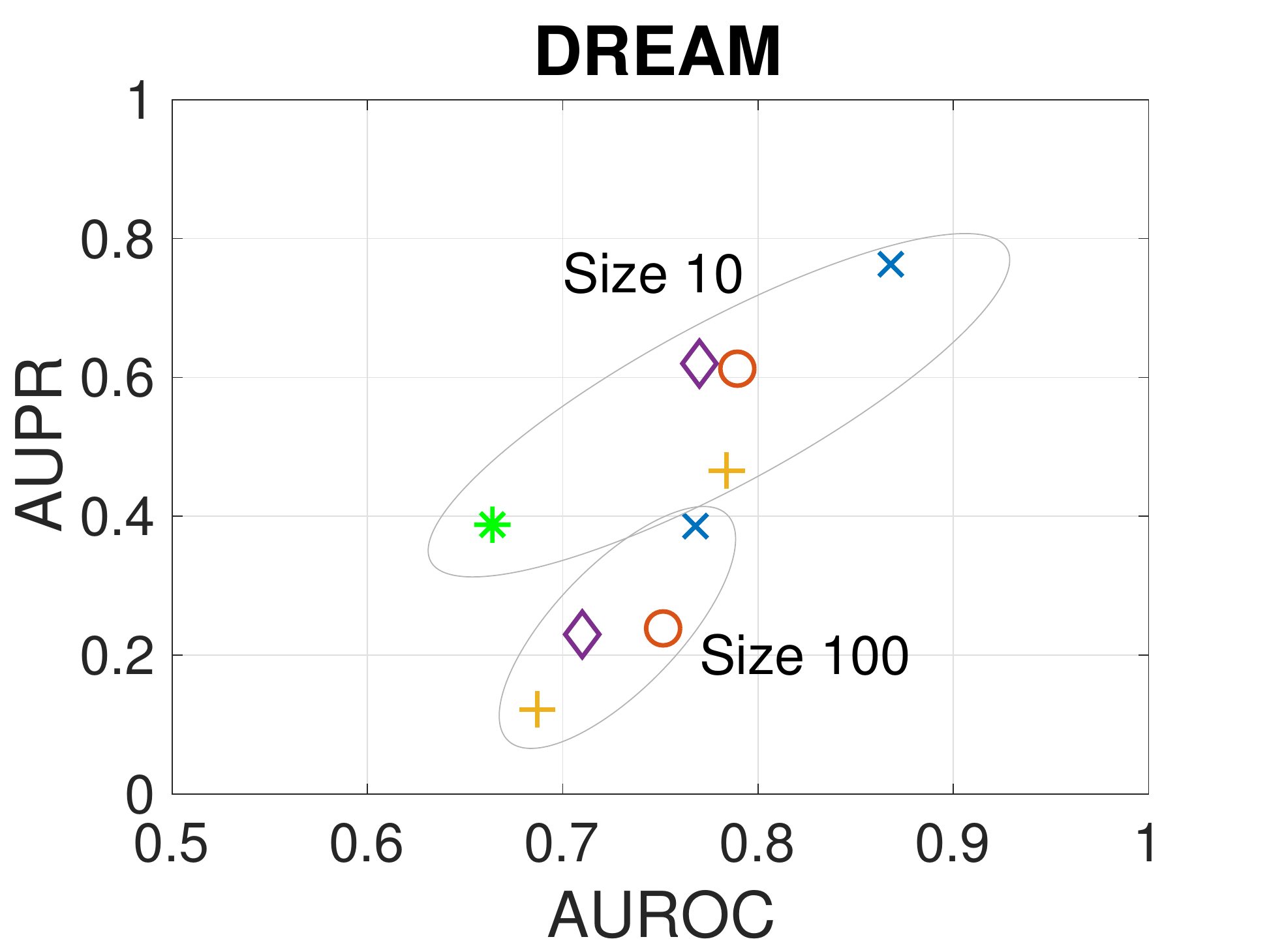}
\hspace{-3mm}
\includegraphics[width=4.5cm]{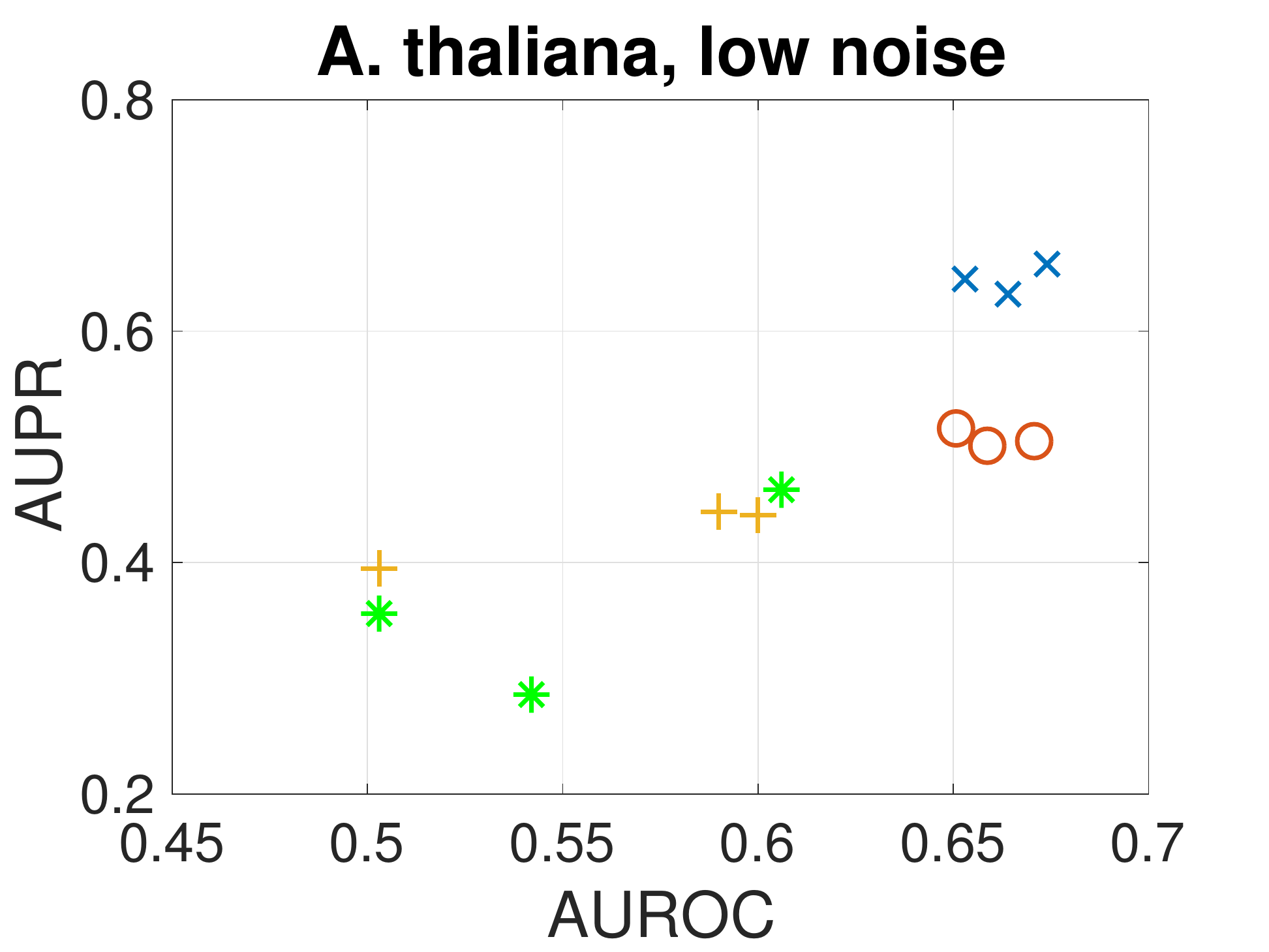}
\hspace{-3mm}
\includegraphics[width=4.5cm]{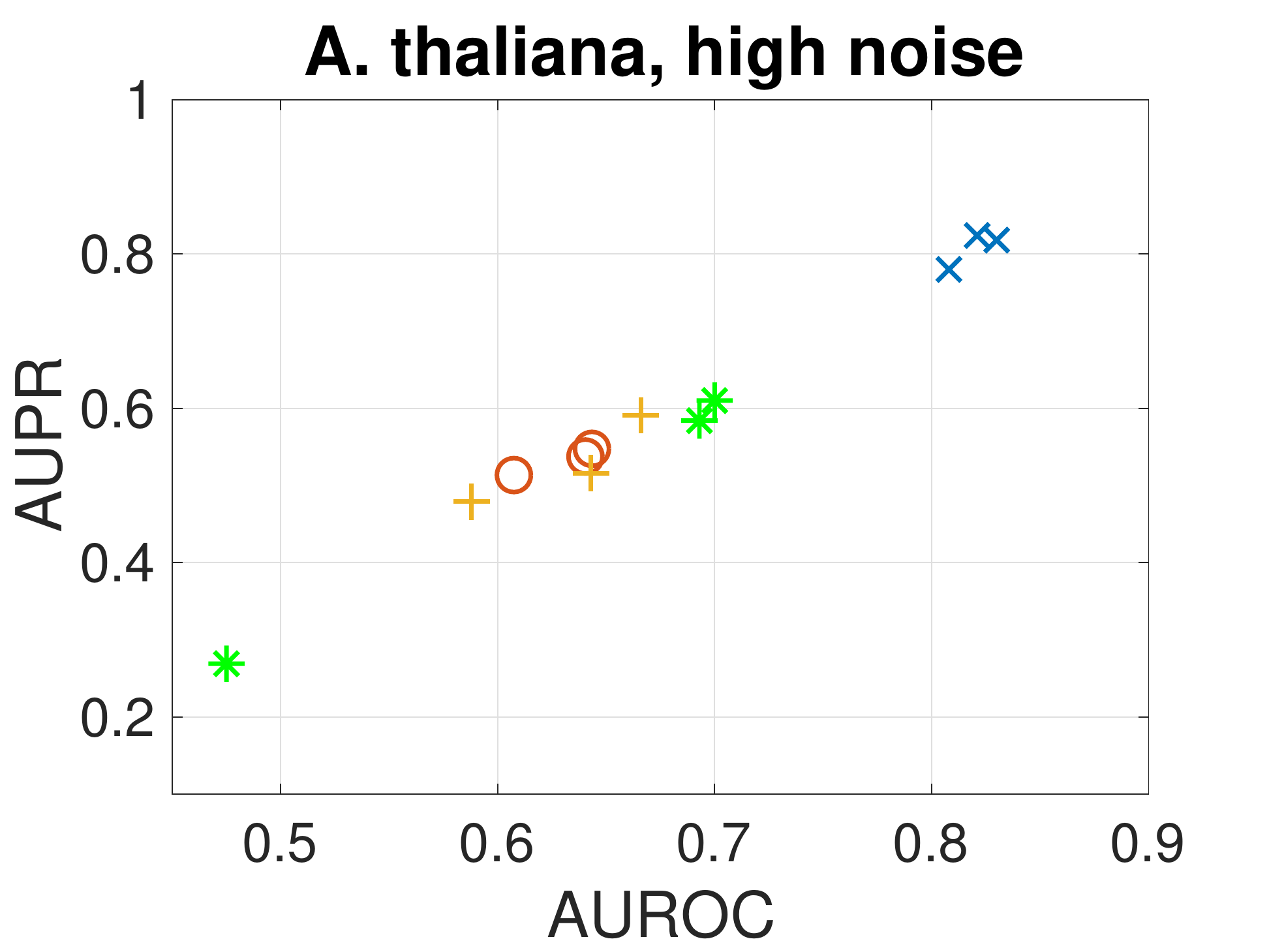}
\hspace{-3mm}
\includegraphics[width=4.5cm]{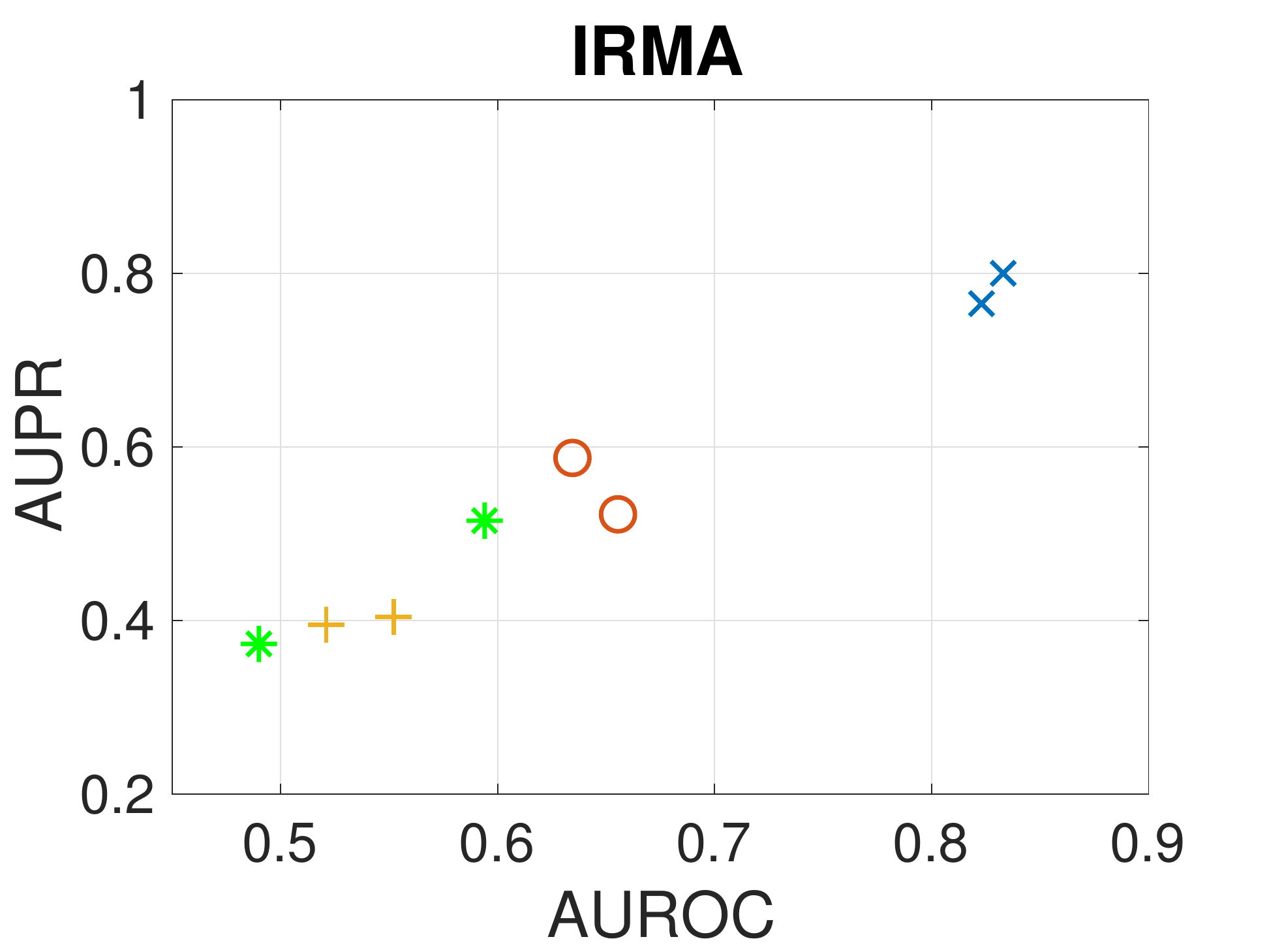}}
\includegraphics[width=11cm]{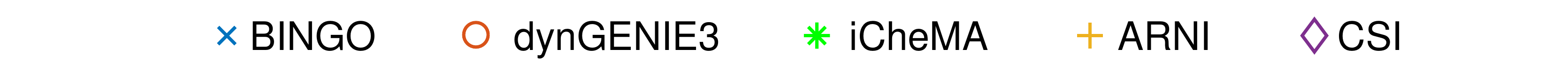}
\caption{Resulting AUROC/AUPR values in different experiments. The DREAM results consist of average values over the five networks using the time series data only. The results on the circadian clock of {\it Arabidopsis thaliana} consist of averages over ten replicates. The results for different sampling rates (1h/2h/4h) are shown separately. In the IRMA results, both full data and averaged data results are shown separately.}
\label{fig:results}
\end{figure}

\subsection{DREAM4 {in silico} network challenge}

The DREAM4 {in silico} challenge consists of network inference tasks with network sizes 10 and 100, with five networks in each size. The data consist of five time series for each 10-gene network and ten time series for each 100-gene network, where different perturbations have been applied on some genes for the first half of the time. The time series illustrate the system's adaptation to the perturbation, and its relaxation when the perturbation is removed. Each time series consists of 21 time points.
In addition, steady state values are provided in the dataset as well as gene knockout and knockdown data corresponding to each gene. For the 10-gene challenge, multifactorial data are provided, which correspond to steady state values under mild perturbations on the basal transcription rate. This corresponds to data collected from different cells, for example.

BINGO is compared with the challenge best performers using all available data, and with other methods using only the time series data. 
The 10-gene challenge winner, \emph{Petri Nets with Fuzzy Logic} (PNFL) is introduced in \cite{PNFL}. The 100-gene challenge winner is introduced in \cite{100winner}. The method is based only on the knockout data, with some post-processing. A similar scoring method without post-processing, the \emph{median-corrected Z-score} (MCZ) method \cite{Combination} achieved the second highest score in the 100-gene challenge.

Methods inferring GRNs from only time series data are reviewed in \cite{Penfold_review}, where the best performer (in terms of average AUPR value) was a method called \emph{Causal Structure Identification} (CSI) \cite{Klemm,CSI_imp}, which is based on Gaussian process regression as well. We include the discrete-time version of CSI in the comparison, since its performance was better.  
As suggested in \cite{Combination}, any method inferring networks from time series data can be combined with a method inferring GRNs from steady state data, such as the MCZ. Unfortunately, the MCZ requires knockouts or knockdowns of all genes, which can hardly be expected in a real experiment. Nevertheless, the combinations dynGENIE3*MCZ and BINGO*MCZ are included in the full data comparison. The scores for the combined methods are the products of the individual scores, favouring links that score high in both methods. It should be noted that BINGO (as well as the PNFL) can utilise also partial knockout data together with time series data. 
The results on the DREAM4 data are summarised in Table~\ref{tab:DREAM_results}.


\begin{table*}[b]
\footnotesize
\vspace{-2mm}
\caption{AUROC/AUPR values for the DREAM4 {in silico} 10-gene (above) and 100-gene (below) network inference challenge data, using either all data or only time series data. The values for PNFL and the 100-gene challenge winner are taken from \cite{DREAM_web}, for dynGENIE3*MCZ from \cite[Suppl. information]{dynGENIE3}, and for CSI from \cite[Table~1]{Penfold_review} (see Remark~\ref{rmk:CSI}). The MCZ method we implemented ourselves, and for the dynGENIE3, iCheMA and ARNI results, the codes provided by the authors of \cite{dynGENIE3}, \cite{Aderhold_mechanistic} and \cite{ARNI}, respectively, were used (see Remark~\ref{rmk:ARNI}).}
    \center
    \mbox{ \hspace{-22mm}
     \begin{tabular}{|l|llr@{\hspace{.4mm}/\hspace{.3mm}}l r@{\hspace{.4mm}/\hspace{.3mm}}l r@{\hspace{.4mm}/\hspace{.3mm}}l r@{\hspace{.4mm}/\hspace{.3mm}}l r@{\hspace{.4mm}/\hspace{.3mm}}l r@{\hspace{.4mm}/\hspace{.3mm}}l|}
    \hline
     \parbox[t]{2mm}{\multirow{12}{*}{\rotatebox[origin=c]{90}{Size 10}}} &
   Data & Method & \multicolumn{2}{c}{Network 1} & \multicolumn{2}{c}{Network 2} & \multicolumn{2}{c}{Network 3} & \multicolumn{2}{c}{Network 4} & \multicolumn{2}{c}{Network 5}   & \multicolumn{2}{c|}{Average} \\
    \cline{2-15}
    & TS & BINGO & {.882} & {.829} & {.790}&{.704} & {.782}&{.567} & {.933}&{.835} & {.954}&{.882} & {.868}&{.763} \\
& & CSI  & (.72) & .64 & (.75)&.54 & (.67)&.45 & (.83)&.67 & (.90)&.78 & (.77)&.62 \\
& & dynGENIE3 & .743 & .551 & .715 & .463 & .765 &.543 & .802 &.706 & .923 &.790 & .790&.611 \\
& & iCheMA & .576 & .401 & .733 & .445 & {.770} & .464 & .563 & .273 & .677 & .357 & .664 & .388 \\
& & ARNI & .835 & .682 & .779 & .626 & .665 & .280 & .768 & .387 & .873 & .355 & .784 & .466 \\ \cline{2-15}
& All &BINGO & .941 & .854 & {.877} & {.779} & .936&.787 & {.957}&{.862} & {.928}&{.830} & .928&{.822} \\
& & PNFL & {.972} &{.916} & .841&.547 & {.990}&{.968} & .954&.852 & .928&.761 & {.937}&.809 \\ & & dynGENIE3*MCZ & NA & .82 & NA&.60 & NA&.80 & NA&.77 &  NA&.59 & NA&.72 \\ 
& & BINGO*MCZ &{.972}&.865&.854&.703&.893&.738&{.966}&{.909}&{.938}&{.846}&.925&.812 \\  
   \cline{2-15}
  &  KO+KD & MCZ &.941&.813&.728&.306&.832&.662&.923&.713&.717&.391&.828&.577 \\ \hline
\multicolumn{15}{c}{ \vspace{-2mm}}  \\   
    \hline
    \parbox[t]{2mm}{\multirow{10}{*}{\rotatebox[origin=c]{90}{Size 100}}} &  
   Data & Method & \multicolumn{2}{c}{Network 1} & \multicolumn{2}{c}{Network 2} & \multicolumn{2}{c}{Network 3} & \multicolumn{2}{c}{Network 4} & \multicolumn{2}{c}{Network 5}   & \multicolumn{2}{c|}{Average} \\ \cline{2-15}
   & TS & BINGO & .816&{.447} & .741&{.296} & .781&{.345} & .787&{.407} & .807&{.438} & .786&{.386}  \\
 & & dynGENIE3 & .789 &.276 & .700 &.175 & .770 &.271 & .736 &.248 & .766 &.214 & .752 &.237 \\
 & & CSI & .71&.25 & .67&.17 & .71&.25 & .74&.24 & .73&.26 & .71&.23 \\
 & & ARNI & .726 &.159 & .641 & .098 & .689 & .109 & .683 & .129 & .696 & .116 & .687 & .122 \\ \cline{2-15}
 & TS+KO \hspace{.8mm} & BINGO & .857 & .485 & .750 & .322 & .796 & .404 & .819 & .435 & .828 & {.456} & .810 & .420 \\
& All & BINGO & .823 & .404 & .725 & .243 & .770 & .299 & .777 & .325 & .788 & .296 & .777 & .313 \\
& & DREAM4 winner & .914&.536  & .801&.377 & .833&.390 & .842&.349 & .759&.213  & .830&.373 \\ & & dynGENIE3*MCZ \hspace{.25mm} & NA&{.60} & NA&{.43} & NA&{.47} & NA&{.52} &  NA&.37 & NA&{.48} \\
 & & BINGO*MCZ &.911&{.588}&.813&.400&.870&{.447}&.856&{.510}&.850&{.464}&.860&{.482} \\ 
    \hline
  \end{tabular}}  
  \label{tab:DREAM_results}
\end{table*}

\subsubsection{The 10-gene network results}

BINGO consistently outperforms other methods by a large margin (with the exception of network 3) in GRN inference from time series data. 
When using all data from the challenge, BINGO scores a little bit higher (average AUPR) than the DREAM4 10-gene challenge winner PNFL. The average scores are very close to each other but in the different networks there are some rather significant differences. BINGO reaches a fairly high AUPR in network 2, which seemed to be very difficult for all challenge participants. The best AUPR for network 2 among the challenge participants was 0.660, and the PNFL's 0.547 was the second highest \cite{DREAM_web}. The poor performance of most methods with network 2 is attributed to low effector gene levels in the wild type measurement~\cite{PNFL}. In contrast, BINGO's performance is less satisfactory with network 3, where the PNFL achieves almost perfect reconstruction. This might be due to a fairly high in-degree (four) of two nodes in the true network. Only one out of eight of these links gets higher confidence value than 0.5 assigned by BINGO. Based on Table~\ref{tab:DREAM_results} and \cite[Table~1]{Penfold_review}, network 3 also seems to be the one where the knockout data has the biggest impact. It may be that the PNFL makes better use of this data. Also the BINGO*MCZ combination scores fairly well with network 3, but in network 2 it loses clearly to BINGO applied to all data directly.


\subsubsection{The 100-gene network results}

As in the 10-gene case, BINGO outperforms its competitors by a clear margin in all five networks when inferring the networks from time series data alone, and in fact, it scores slightly higher than the DREAM4 challenge winner. iCheMA is excluded from this comparison due to its poor scalability to high dimension.

When using all data, the combination BINGO*MCZ is the best performer, tied with the combination dynGENIE3*MCZ. It seems that with the 100-gene network, BINGO cannot always combine different types of data in an optimal way. This may be due to the large number of steady state points where the dynamics function $f$ should vanish. This hypothesis is supported by the fact that the results actually deteriorate when also the knockdown data is included as opposed to using only the knockout data with the time series data. 
 It should be noted that both the DREAM4 winner as well as the MCZ are based solely on the knockout and knockdown data, but their implementation requires knockout of every gene, which is hardly realistic in a real experiment.

\begin{rmk} \label{rmk:CSI}
The AUROC and AUPR values for the method CSI are taken from \cite{Penfold_review}, where the self interactions are included in the computation of these values. The self interactions are given a weight zero, and hence all methods get 10 or 100 ``free" true negatives, depending on the network size. This has some increasing effect on the AUROC values for networks of size 10 (they report mean AUROC of 0.55 for random networks as opposed to 0.5). The effect on the 100-gene network results and on all AUPR values is negligible.
\end{rmk}

\begin{rmk} \label{rmk:ARNI}

In the ARNI method, the user has to choose the type and the order of basis functions. In the DREAM 10-gene case, we tried all basis function sets provided in their Matlab implementation with a variety of orders, and the best performing combinations were tried with the 100-gene case. The best performance overall was achieved with polynomial basis functions with degree 3. The values reported in Table~\ref{tab:DREAM_results}  are obtained with these basis functions. In the article \cite{ARNI}, a method for basis function selection has been introduced, but it was not implemented.

The ARNI method considers a regression problem with input-output pairs $\left(\frac{y_j+y_{j+1}}2,\frac{y_{j+1}-y_j}{\Delta t}\right)$ where $\{y_j\}$ is the time series data. We made a small modification to the implementation by replacing the inputs by $y_j$ which improved the method's performance.

We could not reproduce exactly the dynGENIE3 results for the DREAM4 in silico network inference challenge data reported in \cite{dynGENIE3}. We obtained similar results, but the scores for the different networks varied from the reported scores. Finally, we decided to include results from our own simulations taking into account the perturbations, whereby the results improved slightly. The inputs were incorporated by including five (or ten in the 100-gene case) additional signals to the time series, of which the $j^{\textup{th}}$ signal consisted of 10 ones and 11 zeros in the $j^{\textup{th}}$ experiment, and only zeros in other experiments.

We used the ``random forest" option with $K=n$ in the DREAM4 experiment (as in \cite{dynGENIE3}), but in other experiments we used $K=\sqrt{n}$ which is the default setting in the dynGENIE3 code.
\end{rmk}

\subsection{Circadian clock of {\it Arabidopsis Thaliana}}

Realistic data were simulated from the so-called Millar 10 model of the {\it Arabidopsis thaliana} circadian clock \cite{Millar10}, using the Gillespie method \cite{Gillespie} to account for the intrinsic molecular noise. This model has been widely used to study the plant circadian clock and as a benchmark to assess the accuracy of different network inference strategies \cite{Aderhold_mechanistic}. It simulates gene expressions and protein concentrations time series with rhythms of about 24 hours. The gene regulatory structure consists in a three-loop feedback system of seven genes and their corresponding proteins for which the chemical interactions are described using Michaelis--Menten dynamics. The model has been simulated for 600 hours in 24-hour light/dark cycles to remove all possible transients. Then, the photoperiodic regime was switched to constant light. Ten replicates were simulated and the first 48 hours of the constant light phase was recorded and downsampled to correspond to sampling intervals of 1 hour, 2 hours, or 4 hours. The time series therefore consist of 49, 25, or 13 time points depending on the sampling interval. Two datasets were simulated with different levels of process noise.

Table~\ref{tab:results_Millar} shows the mean AUROC/AUPR values with standard deviations for the methods computed from the ten replicates. BINGO and dynGENIE3 are hardly affected by the decreasing sampling frequency. With less process noise, the AUROC values for these two methods are very close to each other in all cases, but BINGO has somewhat better precision throughout the tested sampling frequencies. With higher process noise, the results of BINGO improve clearly. The iCheMA and ARNI results with 4h sampling rates and 2h sampling rates with low process noise are not much better than random guessing.



\begin{table*}[t]
\caption{Means and standard deviations of AUROC/AUPR values for the simulated circadian clock data with ten replicates.}
    \mbox{ \hspace{-22mm}
  \begin{tabular}{|l|l r@{\hspace{.4mm}/\hspace{.3mm}}l r@{\hspace{.4mm}/\hspace{.3mm}}l r@{\hspace{.4mm}/\hspace{.3mm}}l|}
    \hline
  \parbox[t]{2mm}{\multirow{5}{*}{\rotatebox[origin=c]{90}{Low noise}}} &  Method & \multicolumn{2}{c}{1h sampling} & \multicolumn{2}{c}{2h sampling} & \multicolumn{2}{c|}{4h sampling} \\ \cline{2-8}
& BINGO & {.674} $\pm$ .052 & {.658} $\pm$ .069 & .653 $\pm$ .061 & {.645} $\pm$ .060 & {.664} $\pm$ .060 & {.632} $\pm$ .068 \\
& dynGENIE3 & {.659} $\pm$ .025 & .500 $\pm$ .054 & {.671} $\pm$ .039 & .504 $\pm$ .048 & .651 $\pm$ .043 & .515 $\pm$ .060   \\
& iCheMA & .606 $\pm$ .061 & .463 $\pm$ .068 & .503 $\pm$ .096 & .356 $\pm$ .059 & .542 $\pm$ .120 & .286 $\pm$ .045  \\
& ARNI & .590 $\pm$ .074 & .444 $\pm$ .065 & .600 $\pm$ .069 & .441 $\pm$ .058 & .503 $\pm$ .055 & .395 $\pm$ .062 \\
    \hline
\multicolumn{8}{c}{ \vspace{-2mm}}  \\   
    \hline    
    \parbox[t]{2mm}{\multirow{5}{*}{\rotatebox[origin=c]{90}{High noise}}} &  Method & \multicolumn{2}{c}{1h sampling} & \multicolumn{2}{c}{2h sampling} & \multicolumn{2}{c|}{4h sampling} \\ \cline{2-8}
& BINGO & {.821} $\pm$ .040 & {.824} $\pm$ .037 & {.830} $\pm$ .032 & {.818} $\pm$ .040 & {.808} $\pm$ .035 & {.780}  $\pm$ .038  \\
& dynGENIE3 & .641 $\pm$ .027 & .536 $\pm$ .027 & .644 $\pm$ .039 & .546 $\pm$ .048 & .608 $\pm$ .101  & .512 $\pm$ .091   \\
& iCheMA & .693 $\pm$ .054 & .584 $\pm$ .045 & .700 $\pm$ .061  & .610 $\pm$ .055 & .475 $\pm$ .098  &  .269 $\pm$ .041   \\
& ARNI & .666 $\pm$ .051 & .591 $\pm$ .071 & .643 $\pm$ .074 & .516 $\pm$ .080 & .588 $\pm$ .055 & .479 $\pm$ .094 \\
    \hline
  \end{tabular}}
  \label{tab:results_Millar}
  \end{table*}

%

\subsection{{In vivo} dataset IRMA}

A synthetic network was constructed in \cite{IRMA} with the purpose of creating an {in vivo} dataset with known ground truth network for benchmarking network inference and modelling approaches. The network is rather small, consisting of only five genes and eight links in the ground truth network. Nevertheless, this dataset can be used to verify the performance of BINGO  using real data.

The IRMA network can be ``switched on" and ``off" by keeping the cells in galactose or glucose, respectively. The dataset consists of nine transient time series, where the network is either switched on (five time series) or off (four time series) at the beginning of the experiment. These have been averaged into one switch-on time series (with 16 time points with 20 minute sampling interval) and one switch-off time series (with 20 time points with 10 minute sampling interval). Typically only the two average time series have been used, but we try BINGO with both the two average time series, as well as with all the nine experiments separately.

 \begin{table}[t]
  \caption{AUROC/AUPR values for the IRMA dataset using either the two averaged time series or all nine time series.}
  \begin{tabular}{|l r@{\hspace{.4mm}/\hspace{.3mm}}l r@{\hspace{.4mm}/\hspace{.3mm}}l|}
    \hline
    Method & \multicolumn{2}{c}{Avg. data} & \multicolumn{2}{c|}{Full data} \\ \hline
 BINGO & .833 & .800  & .823 & .765 \\
 dynGENIE3 & .635 & .586 & .656 & .521 \\
 iCheMA & .490 & .373 & .594 & .515 \\
 ARNI & .521 & .395 & .552 & .404 \\
    \hline
  \end{tabular}
  \label{tab:results_IRMA}
\end{table}

The results in terms of AUROC/AUPR scores are presented in Table~\ref{tab:results_IRMA}. Moreover, the dataset has been used in other recent articles presenting methods ELM-GRNNminer \cite{ELM}, and the TimeDelay-ARACNE \cite{TD-ARACNE}. However, they only report one network structure as opposed to a list of links with confidence scores. Therefore it is not possible to calculate AUROC/AUPR scores for these methods, but it is possible to represent their predictions as points with the ROC and precision-recall curves obtained for BINGO and dynGENIE3, presented in Figure~\ref{fig:IRMA}. With such small network, the AUROC and AUPR values are very sensitive to small differences in predictions. The best predicted network using the averaged data has five out of eight links correct, and one false positive. The best predictions from the dynGENIE3 with the same data have either four correct links with one false positive or five correct links with three false positives. However, it is not evident if these best predictions can be concluded from the results. With BINGO it is possible to look at the histogram of the posterior probabilities of all possible links,  shown in Figure~\ref{fig:IRMA_histogram}. In the averaged data case, the best prediction with five true links with one false positive stands out relatively well. Using the full data, there are three false positives that get confidence of over 0.9.


\begin{figure}[t]
\center
\mbox{\hspace{-5mm}
\includegraphics[width=13cm]{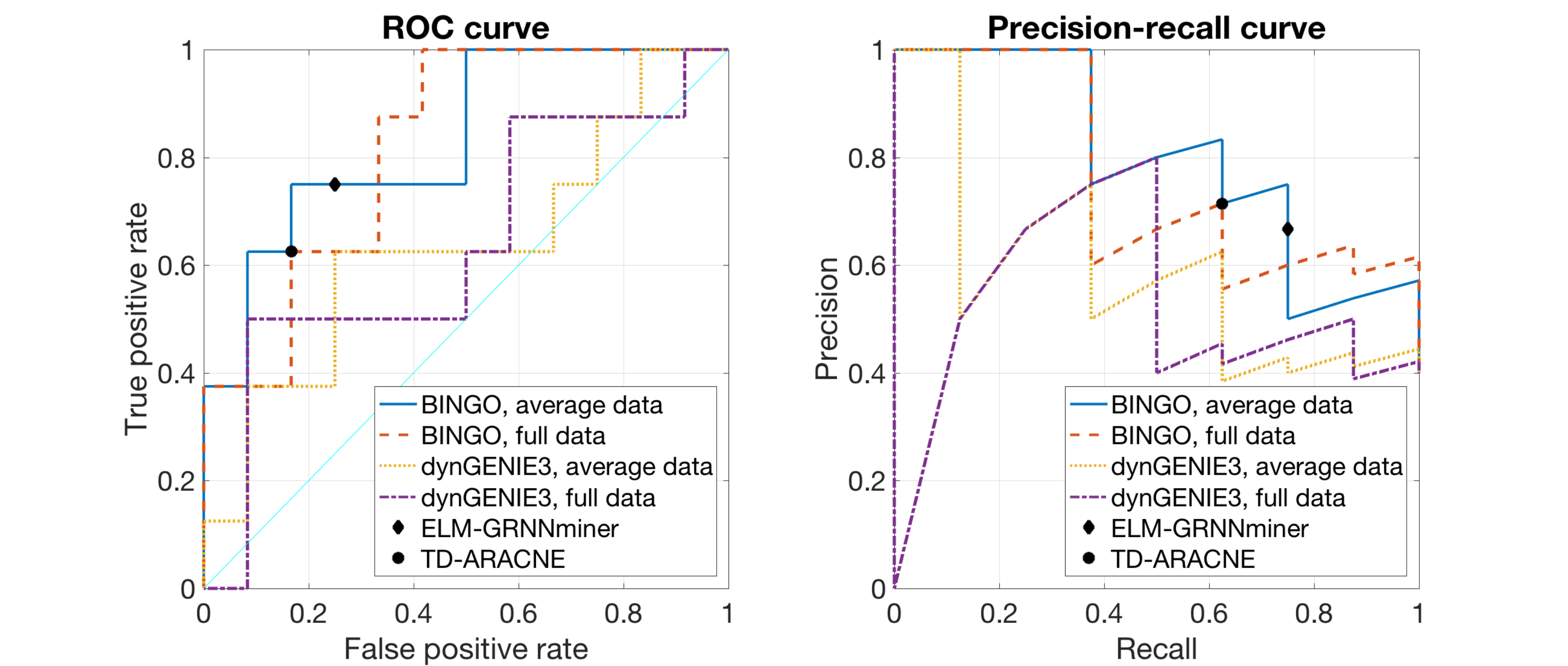} }
\caption{The ROC and precision-recall curves for BINGO and dynGENIE3 using either all nine time series, or the two averaged time series together with predictions from the ELM-GRNNminer (obtained from \cite[Figure~4]{ELM}) and the TD-ARACNE (from \cite[Figure~5]{TD-ARACNE}).}
\label{fig:IRMA}\end{figure}

\begin{figure}[t]
\center
\mbox{\hspace{-2mm}
\includegraphics[width=6cm]{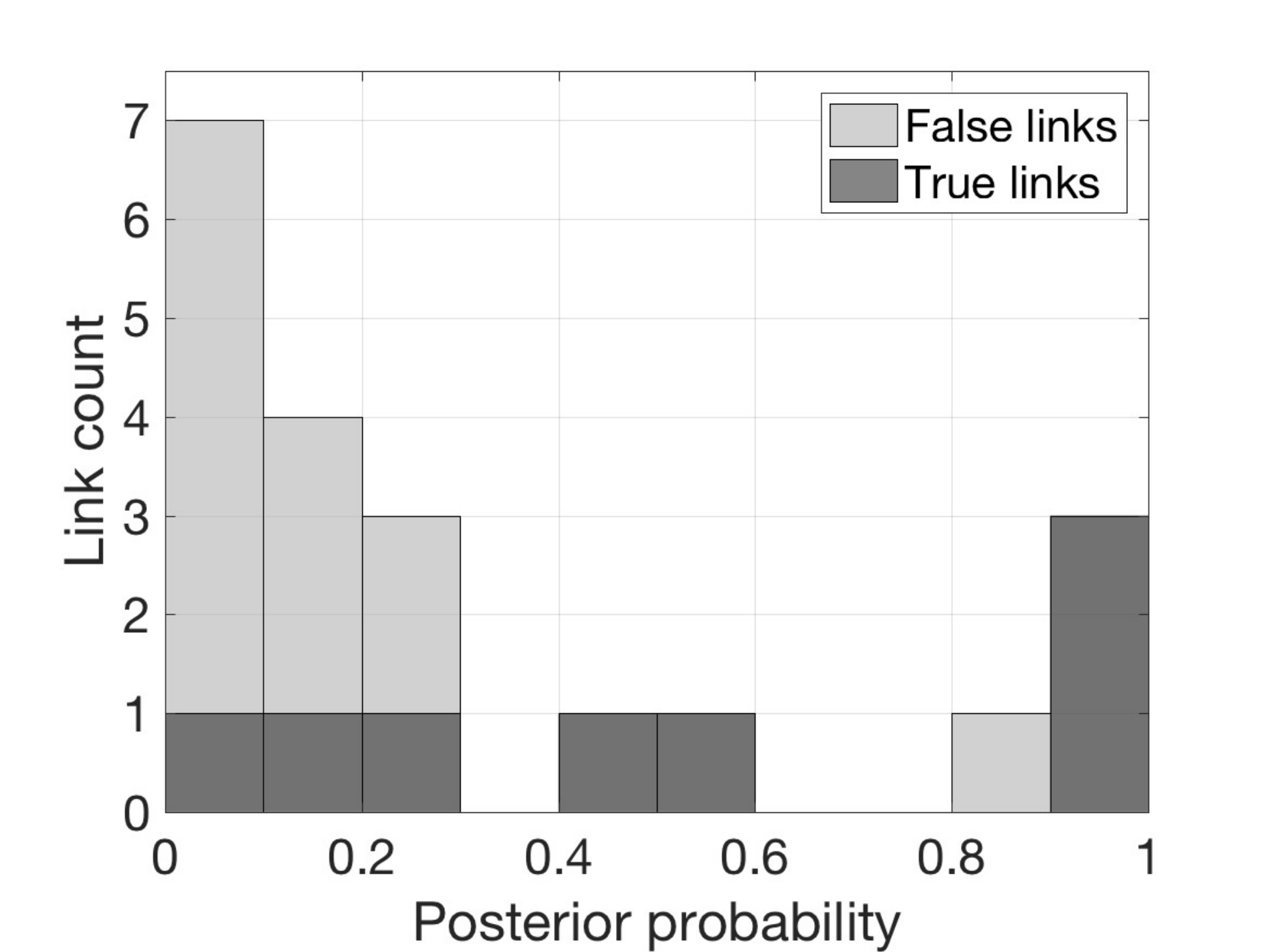} 
\hspace{-5mm}
\includegraphics[width=6cm]{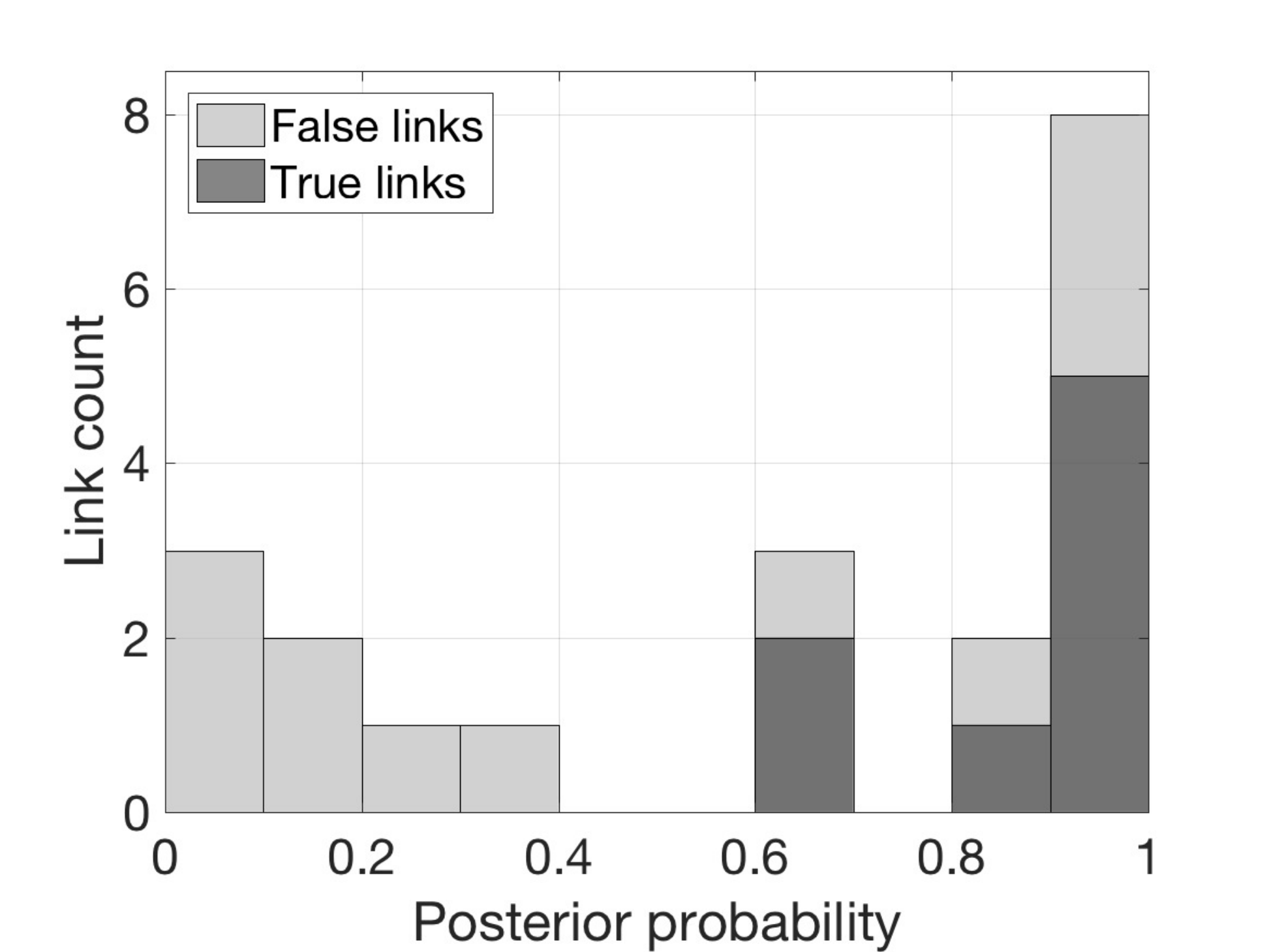} \hspace{-2mm}}
\caption{The histograms of posterior probabilities of all links for the averaged (left) and full (right) IRMA data. In the average data case, the prediction with three true positives and no false positives is obtained with threshold between 0.86 and 0.97. The best prediction with five true and one false positive is obtained with threshold between 0.25 and 0.45. In the full data case, the likely predictions are six true and four false positives with threshold between 0.66 and 0.85, and eight true positives (all) and five false positives with threshold 0.39 and 0.60.}
\label{fig:IRMA_histogram}
\end{figure}

\begin{rmk} \label{rmk:ELN}

In the analysis of the results, we have ignored self-regulation as was done in the DREAM4 challenge. Therefore, in the IRMA network, the maximum number of links is 20 (in \cite{TD-ARACNE,ELM} it is 25). Moreover, it seems that in \cite{ELM} one link (Gal4 $\to$ Swi5) has been omitted from the ground truth. With our criteria, the ELM-GRNNminer had 6 true positives and 3 false positives, and TD-ARACNE had 5 true positives and 2 false positives.

\end{rmk}

\section{Discussion} \label{sec:discussion}

A nonparametric method BINGO for gene regulatory network inference was presented, which is based on the continuous-time Gaussian process dynamical model. We also presented theory behind the continuous-time GPDM. The Gaussian process framework has proven very useful in nonlinear regression problems arising in machine learning. Due to the analytical tractability of Gaussian processes, it is possible to obtain a probability distribution for the trajectories of the GPDM. Such distribution allows MCMC sampling of the continuous trajectories, thereby bypassing a caveat of estimating derivatives from time series with low sampling frequency---a far too common procedure in existing GRN inference methods.

BINGO was favourably compared to state-of-the-art methods in GRN inference from time series data in various examples. In particular, it was demonstrated that the approach based on sampling continuous gene expression trajectories is good for handling time series data with low sampling frequency. Moreover, it was shown that the method can integrate steady state data with time series data to improve performance. BINGO was also successfully applied on real biological data.

BINGO is computationally heavier than dynGENIE3, for example, which is among the best methods in terms of scalability to large dimensions. However, given the time, effort, and cost of a gene expression experiment, the computation time is hardly as important as the accuracy of predictions, as long as the method is scalable to high enough system dimension. A MCMC approach is perfectly parellelisable: independent chains are run on different processors, and the collected samples are pooled together in the end. Parallelisation allows inference of networks of even a couple of thousands of variables. To test scalability, BINGO was applied on a data with dimension 2000, consisting of five time series of 21 time points each. With this size, network inference could be carried out overnight (see Remark~\ref{rmk:scalability}).

Recently developed so-called single-cell techniques enable gene expression measurements in one cell resolution for a large number of cells at a time. The cell is destroyed in the measurement process, and therefore the data consist of ensemble snapshots rather than time series. It is possible to obtain so-called pseudotime series from such data \cite{monocle,GP_pseudotime}, and BINGO can be used on such time series---although a small modification in fitting the trajectory samples to the measurements is required, due to the large amount of measurements typically obtained from single-cell measurements. The method can also be integrated with a pseudotime estimator, but this is left for future development.

Interesting future research topics include applying BINGO to solve different biological and biomedical real-data problems. From theoretical perspective, it would be desirable to relax smoothness requirements and to consider process noise with memory and/or dependence on the system's state, which is also more realistic from the application point of view \cite{Gillespie}.

\appendix

\section{Efficient sampling schemes} \label{app:sampling}

\subsection{Pseudo-input scheme}
Gaussian process regression suffers from a very unfavourable scaling of the computational load with respect to the number of data points. This problem is further aggravated by our scheme, where the number of data points used in the GP regression is in fact the number of discretisation points in the continuous time trajectory. However, we can resort to a pseudo-input scheme, where this scaling becomes linear.

In the pseudo-input scheme \cite{pseudo_input}, the underlying Gaussian process $f$ is characterised through so-called pseudo-data $P:=\{(\bar x_j,\bar f_j)\}_{j=1}^p$, where $\bar f_j = f(\bar x_j)$.  The number of pseudo-inputs $p$ is specified by the user, based on the available computing power and the size of the original problem. The pseudo-inputs are not related to the inputs of the actual data, but instead they can be considered as hyperparameters, and they can be estimated by a maximum likelihood approach or they can be sampled as well. Another approach is to use only a subset of the actual input-output data (a so-called active set) in the regression \cite{pseudo_input_Seeger}. We use the pseudoinput approach of \cite{pseudo_input}, but the main idea then follows \cite{pseudo_input_Seeger}, that is the value $f(x)$ at a generic point $x$ is approximated by $\Ex(f(x)|P)$. When the pseudo-outputs $\bar f_j$ are integrated out, the approximation leads to replacement of the matrices $K_i(\uX)$ in \eqref{eq:pX} by
\[
K_i(\uX) \approx K_i(\uX,P)K_i(P)^{-1}K_i(\uX,P)^{\top},
\]
where $K_i(\uX,P) \in \mathbb{R}^{M \times p}$ is a matrix whose element $(j,k)$ is $k_i(X_{\tau_{j-1}},\bar x_k)$. Similarly $K_i(P) \in \mathbb{R}^{p \times p}$ is a matrix whose element $(j,k)$ is $k_i(\bar x_j,\bar x_k)$. The approximation used in \cite{pseudo_input} is more accurate, but its computational cost is much higher when it is not used only for regression.

With this approximation, it is possible to use the Woodbury identity and the matrix determinant lemma again to obtain for the exponent in \eqref{eq:pX}
\begin{align*}
& \big(\Delta\tau K_i(\uX,P)K_i(P)^{-1}K_i(\uX,P)^{\top} \Delta\tau + q_i\Delta\tau \big)^{-1} \\
& = (q_i\Delta\tau)^{-1} - \frac1{q_i}K_i(\uX,P) \big( q_i K_i(P) + K_i(\uX,P)^{\top}\Delta\tau K_i(\uX,P) \big)^{-1}K_i(\uX,P)^{\top}. 
\end{align*}
Here $q_i\Delta\tau$ is a diagonal matrix and the full matrix inverse is computed for a $p \times p$ matrix instead of $M \times M$.
The downside is that the determinant term becomes
\begin{align*}
& \big|\Delta\tau K_i(\uX,P)K_i(P)^{-1}K_i(\uX,P)^{\top} \Delta\tau + q_i\Delta\tau \big| \\
& = |K_i(P)|^{-1} |q_i\Delta\tau| \Big| K_i(P)+\frac1{q_i}K_i(\uX,P)^{\top}\Delta\tau K_i(\uX,P) \Big|
\end{align*}
where $|K_i(P)|$ must be computed separately. Notice, for example, that $|K_i(P)|$ tends to zero if two pseudo-inputs tend to each other, so it has an effect of pushing the pseudo-input points apart from each other. In practical implementation, a small increment $\varepsilon I$ is added to the matrix $K_i(P)$ to ensure numerical stability. This corresponds to assuming that the pseudo-outputs $\bar f_j$ are corrupted by small noise (with variance $\varepsilon I$). We sample the pseudoinputs using random walk sampling, using a uniform prior for the pseudoinputs in the hypercube covering the actual data.

\subsection{Crank--Nicolson sampling} \label{sec:CN}

In the presented algorithm, the discretised trajectory $X$ is sampled using MCMC. When the discretisation is refined, the acceptation rate tends to decrease when conventional samplers are used. This can be avoided by Crank--Nicolson sampling \cite{Beskos_CN,CN_MCMC}, if the target distribution has a density with respect to a Gaussian measure,
\[
p(x)=\Phi(x)\mathcal{N}(x;m,P).
\]
The Crank--Nicolson sampling then works as follows. Assume the current sample is $x^{(l)}$. The candidate sample is $\hat x=m+\sqrt{1-\varepsilon^2}(x^{(l)}-m)+\varepsilon\xi$, where $\xi \sim \mathcal{N}(0,P)$. The new sample is then accepted with probability $\min\big\{1,\Phi(\hat x)/\Phi(x^{(l)}) \big\}$. The step length parameter $\varepsilon \in (0,1)$ is chosen by the user.

Crank--Nicolson sampling plays well along with the pseudo-input scheme. The term $(q_i\Delta\tau)^{-1}$ in the matrix inverse approximation above, and the term $|q_i\Delta\tau|$ in the determinant correspond exactly to the Wiener measure on the discretised trajectory. Notice that also the data fit term $p(Y|x,\theta)$ is Gaussian. In order to get a sampler producing reasonable trajectory candidates, we factorise the Wiener measure
\[
\mathcal{W}(dx)=\prod_{j=1}^m \mathcal{N}\big(x_{t_j}-x_{t_{j-1}};0,Q(t_j-t_{j-1})\big)\mathcal{B}_{(t_{j-1},t_j)}(dx),
\]
where $\mathcal{B}_{(t_{j-1},t_j)}(dx)$ is the Brownian bridge measure on interval $(t_{j-1},t_j)$, that is fixed to values $x_{t_{j-1}}$ and $x_{t_j}$ at the end points. Finally, the Gaussian measure that is used in the Crank--Nicolson sampler is
\[
\mathcal{N}(Y|x,\theta)\prod_{j=1}^m \mathcal{B}_{(t_{j-1},t_j)}(dx),
\]
and the factors $\prod_{j=1}^m \mathcal{N}\big(x_{t_j}-x_{t_{j-1}};0,Q(t_j-t_{j-1})\big)$ are implemented in the acceptance probability.

\section{Integration of the exponential function} \label{app:expint}

Consider the integral
\[
\int_{\mathbb{R}^N} \exp(-J(x))dx
\]
where
\[
J(x)=\ip{x,Ax}+\ip{b,x}+c,
\]
and $A$ is symmetric and positive definite. Now $J$ can be written as 
$$
J(x)=J_{\min}+\ip{x-x_{\min},A(x-x_{\min})}
$$ 
where $J_{\min}=\min_x J(x)$ and $x_{\min}$ is the (unique) vector attaining this minimum. Then
\begin{align*}
\int_{\mathbb{R}^N} \exp(-J(x))dx &= \exp(-J_{\min}) \int_{\mathbb{R}^N} \exp\big( -\ip{x-x_{\min},A(x-x_{\min})} \big) dx \\ & = \exp(-J_{\min}) \int_{\mathbb{R}^N} \exp\big( -\ip{x,Ax} \big) dx \\
& = \frac{\pi^{N/2}}{|A|^{1/2}} \exp(-J_{\min}).
\end{align*}
Finally, the minimum is
\[
J_{\min}=c-\frac14 \ip{b,A^{-1}b}.
\]
In the derivation of $p(X|\theta)$, this is applied so that
\[
\begin{cases}
A= \frac1{2q_i}\Delta\tau + \frac12 K_i(\uX)^{-1}, \\
b=-\frac1{q_i}(\oX_i-\uX_i), \\
c=\frac1{2q_i} \norm{\oX_i-\uX_i}_{\Delta \tau^{-1}}^2.
\end{cases}
\]

\section{Remarks on the implementation of BINGO} \label{sec:implementation}

Some details of the numerical examples are presented in Table~\ref{tab:simulations}. In the experiments, the time series were scaled so that the difference of the maximal and minimal expression value for each gene was one, so that parameter priors would be consistent across dimensions. The scaling is not completely necessary, since the priors are either scale free, or are scaled accordingly if either the data is scaled or the time axis is scaled. The priors for the parameters are as follows:
\begin{itemize}
\item Noninformative inverse gamma prior for the process noise covariance $q_i$, measurement noise covariance $r_i$, and the steady state covariance $M_{i,ss}$
\[
p(q_i) \propto \frac1{q_i^{1.001}}\exp\left(-\frac{0.00001}{q_i}\right),
\]
\item Exponential priors for $a_i \ge 0$, $b_i \ge 0$, and $\beta_{i,j}$
\[
p(a_i) \propto \exp\left(-\frac{a_i}{10V(Y_i)}\right), \quad p(b_i) \propto \exp\left(-\frac{b_i}{5V(Y_i)}\right), \quad p(H_{i,j}) \propto \exp\left(-\frac{H_{i,j}}{\textup{ran}(Y_j)} \right),
\]
where $V(Y_i)$ is the variation of $i^{\textup{th}}$ component of the trajectory per time unit (approximated from data), and $\textup{ran}(Y_j)$ is the range of the $j^{\textup{th}}$ trajectory:
\[
V(Y_i) = \frac1{t_m-t_0}\sum_{j=1}^m |[y_j]_i-[y_{j-1}]_i| \qquad \textup{and} \qquad \textup{ran}(Y_j)=\max_k [y_k]_j-\min_k [y_k]_j.
\]
Note that $\textup{ran}(Y_j)=1$ if the time series are scaled as described above.

\item Gamma prior (truncated) for $\gamma_i$
\[
p(\gamma_i) \propto \gamma_i \exp\left(-\frac{\gamma_i}{5\sigma(\Delta Y_i)}\right) (30-\gamma_i/\sigma(\Delta Y_i))
\]
where $\sigma(\Delta Y_i)$ an estimate of the variance of the derivative of the $i^{\textup{th}}$ component of the trajectory:
\[
\sigma(\Delta Y_i) = \frac1{m}\sum_{j=1}^m \left(\frac{[y_j]_i-[y_{j-1}]_i}{t_j-t_{j-1}}\right)^2.
\]
\item Inverse gamma prior for the knockout measurement covariance
\[
p(M_{i,ko}) \propto \frac1{M_{i,ko}^{N_{i,ko}/2}}\exp\left(-\frac{N_{i,ko}\sigma(\Delta Y_i)}{10M_{i,ko}}\right)
\]
where $N_{i,ko}$ is the number of knockout measurements taken into account when inferring links pointing to gene $i$.

Ideally also $M_{i,ko}$ should have a noninformative prior, but it was observed that this variable had a tendency to become either very small, thereby giving all weight to the knockout data and neglecting the time series data, or very large with the opposite effect. This might be due to some mismatch in the time series data and the knockout data. Nevertheless, using all data simultaneously still seemed to produce best results, but in order to achieve a good balance between both data types, the values for $M_{i,ko}$ have to be forced to a good range using an informative prior like this.
\end{itemize}

\begin{table}[t]
\footnotesize
\vspace{-2mm}
\caption{\footnotesize Simulation details on the benchmark examples. In DREAM4 size 100, three independent chains were run in parallel. The total number of sampling rounds is the burn-in length added to the number of samples multiplied by the thinning factor. The computational times are for inferring one network. They are obtained with a Macbook pro, 2.4 GHz Intel Core i7, except for the DREAM4, size 100 case, which is with Dell, 2.5 GHz Intel Xeon E5-2680 v3.}
    \mbox{
  \begin{tabular}{|l c c c c c|}
    \hline
    Experiment & $\eta$ & Burn-in  & Number of & thinning & computation \\
    & & & samples & factor & time (min) \\ \hline
 DREAM4, size 10 & 1/10 & 3000 & 10000 & 10 & 31 \\
  DREAM4, size 100 & 1/100 & 1500 & 3 $\times$ 3000 & 10 & 3 $\times$ 188 \\
 Circadian clock & 1/7 & 3000 & 6000 & 10 & (1h/2h/4h)  6/5/4     \\
 IRMA & 1/5 & 3000 & 10000 & 10 & (avg./full) 7/18 \\
    \hline
  \end{tabular}}
  \label{tab:simulations}
\end{table}

\begin{rmk} \label{rmk:scalability}
To test the BINGO's scalability, it was tried on the dataset obtained by concatenating the DREAM4 size 10 time series 200 times to obtain a dataset consisting of five time series with 21 time points with dimension 2000. With Macbook pro, 2.4 GHz Intel Core i7, it took 592 seconds to collect 50 samples with a discretisation level three times finer than the measurement discretisation. Parallelising to 20 processors with similar capacity, it would take nine hours to collect 5000 samples (with burn-in of 250 samples per chain, and thinning factor of 10). 
\end{rmk}


\end{document}